\documentclass{amsart}
\usepackage{array}
\newcolumntype{P}[1]{>{\raggedright\let\newline\\\arraybackslash\hspace{0pt}}m{#1}}
\usepackage{graphicx,verbatim, amsmath, amssymb, amsthm, amsfonts, epsfig, amsxtra,ifthen,mathtools,
caption,enumerate,hyperref,hhline,bbm,capt-of,longtable}	
\DeclareFontFamily{U}{mathx}{\hyphenchar\font45}
\DeclareFontShape{U}{mathx}{m}{n}{<-> mathx10}{}
\DeclareSymbolFont{mathx}{U}{mathx}{m}{n}
\DeclareMathAccent{\widebar}{0}{mathx}{"73}

\newtheorem{proposition}{Proposition}[section]
\newtheorem{theorem}[proposition]{Theorem}
\newtheorem{corollary}[proposition]{Corollary}
\newtheorem{lemma}[proposition]{Lemma}
\newtheorem{prop}[proposition]{Proposition}
\newtheorem{cor}[proposition]{Corollary}
\newtheorem{thm}[proposition]{Theorem}

\theoremstyle{definition}
\newtheorem{example}[proposition]{Example}

\theoremstyle{remark}
\newtheorem{remark}[proposition]{Remark}

\numberwithin{equation}{section}
\usepackage[usenames]{color}

\usepackage{color}

\newcommand{\margincolor}{red}      
\definecolor{darkgreen}{rgb}{0,0.7,0}

\newcounter{margincounter}
\newcommand{\marginnum}{
\ifnum\value{margincounter}<10
\textcolor{\margincolor}{\begin{picture}(0,0)\put(2.2,2.4){\circle{9}}\end{picture}\footnotesize\arabic{margincounter}}
\else\ifnum\value{margincounter}<100
\textcolor{\margincolor}{\begin{picture}(0,0)\put(4.256,2.5){\circle{11}}\end{picture}\footnotesize\arabic{margincounter}}
\else
\textcolor{\margincolor}{\begin{picture}(0,0)\put(6.8,2.5){\circle{14}}\end{picture}\footnotesize\arabic{margincounter}}
\fi\fi
}

\newcommand{\newword}[1]{\textbf{\emph{#1}}}

\newcommand{\integers}{\mathbb Z}
\newcommand{\rationals}{\mathbb Q}

\newcommand{\reals}{\mathbb R}

\newcommand{\ep}{\varepsilon}
\newcommand{\thet}{\vartheta}

\newcommand{\Ram}{{\operatorname{Ram}}}
\newcommand{\uf}{{\operatorname{uf}}}
\newcommand{\fr}{{\operatorname{fr}}}

\newcommand{\cov}{\mathrm{cov}}

\newcommand{\set}[1]{{\left\lbrace #1 \right\rbrace}}

\newcommand{\br}[1]{{\langle #1 \rangle}}
\newcommand{\A}{{\mathcal A}}

\newcommand{\F}{{\mathcal F}}
\newcommand{\D}{{\mathfrak D}}
\newcommand{\N}{{\mathcal N}}
\newcommand{\p}{{\mathfrak p}}

\newcommand{\ck}{\spcheck}

\newcommand{\0}{{\mathbf{0}}}

\newcommand{\Cone}{\mathrm{Cone}}
\newcommand{\Star}{\mathrm{Star}}
\newcommand{\Lin}{\mathrm{Lin}}

\DeclareMathOperator{\Span}{Span}

\newcommand{\g}{\mathbf{g}}
\newcommand{\s}{\mathbf{s}}
\newcommand{\m}{\mathbf{m}}
\renewcommand{\c}{\mathbf{c}}

\renewcommand{\k}{\mathbbm{k}}

\renewcommand{\a}{\mathbf{a}}

\newcommand{\tB}{\tilde{B}}

\newcommand{\M}{\mathcal{M}}

\newcommand{\Supp}{\operatorname{Supp}}

\newcommand{\Clear}{\operatorname{Clear}}

\newcommand{\Hom}{\operatorname{Hom}}
\newcommand{\Scat}{\operatorname{Scat}}
\newcommand{\Fan}{\operatorname{Fan}}
\newcommand{\ScatFan}{\operatorname{ScatFan}}
\newcommand{\ChamberFan}{\operatorname{ChamberFan}}
\newcommand{\gFan}{\g\!\operatorname{Fan}}
\newcommand{\CambScat}{\operatorname{CambScat}}
\newcommand{\Nar}{\operatorname{Nar}}

\newcommand{\re}{\mathrm{re}}
\newcommand{\im}{\mathrm{im}}
\renewcommand{\d}{{\mathfrak d}}
\newcommand{\seg}[1]{\overline{#1}}
\newcommand{\hy}{\hat{y}}

\allowdisplaybreaks

\title{A combinatorial approach to scattering diagrams}
\author{Nathan Reading}
\thanks{Partially supported by the National Science Foundation under Grant Number DMS-1500949.}

\begin{document}

\begin{abstract}
Scattering diagrams arose in the context of mirror symmetry, but a special class of scattering diagrams (the cluster scattering diagrams) were recently developed to prove key structural results on cluster algebras.
We use the connection to cluster algebras to calculate the function attached to the limiting wall of a rank-2 cluster scattering diagram of affine type.
In the skew-symmetric rank-2 affine case, this recovers a formula due to Reineke.
In the same case, we show that the generating function for signed Narayana numbers appears in a role analogous to a cluster variable.
In acyclic finite type, we construct cluster scattering diagrams of acyclic finite type from Cambrian fans and sortable elements, with a simple direct proof.
\end{abstract}
\maketitle


\setcounter{tocdepth}{1}
\tableofcontents


\section{Introduction}\label{intro sec}  
In this paper, we demonstrate cluster-algebraic and Coxeter/root-theoretic approaches to the construction of cluster scattering diagrams, and prove results that relate cluster scattering diagrams to classical objects in algebraic combinatorics.
A crucial ingredient is the connection made in~\cite{GHKK} between scattering diagrams and cluster algebras.
We begin by outlining an approach to scattering diagrams from the direction of Coxeter groups/root systems.
A standard approach can be found in~\cite{GHKK,KS}, while~\cite{scatfan} serves as a bridge between the two approaches.

Among the defining data for either a cluster algebra or a scattering diagram is a skew-symmetrizable integer matrix $B$ called an exchange matrix.   
In Section~\ref{rk2 sec}, we consider the special case where the rank of $B$ is~$2$.
Results of \cite{GHKK} and easy, known formulas for $\g$-vectors in rank $2$ reveal the entire cluster scattering diagram in the rank-$2$ \emph{affine} case except for the function attached to the ``limiting'' wall.
We compute this function in Theorem~\ref{rk2 aff formula}.
The skew-symmetric case of Theorem~\ref{rk2 aff formula} was established using representation theory by Reineke~\cite[Section~6]{Reineke2}.
The non-skew-symmetric case of Theorem~\ref{rk2 aff formula} may be new.

Our proof of Theorem~\ref{rk2 aff formula} expresses the function attached to the limiting wall as a limit of ratios of (powers of) adjacent $F$-polynomials, using the key observation that cluster variables can be obtained as path-ordered products.
We also calculate, in the skew-symmetric affine case, a path-ordered product related to this key observation.
Specifically, the formula for a cluster variable in terms of a path-ordered product takes as input the $\g$-vector of the cluster variable.
We compute the analogous path-ordered product for the ``limiting'' $\g$-vector, which is not the $\g$-vector of a cluster variable, and find a surprising appearance (Theorem~\ref{Nar thm}) of the classical Narayana numbers.
The proof of Theorem~\ref{Nar thm} shows that this path-ordered product is the limit of ratios of adjacent cluster variables and exploits a symmetry of the cluster algebra to establish a functional equation for the limit.
We conclude our discussion of rank $2$ with some examples of the computation of theta functions.

In Section~\ref{camb sec}, we construct cluster scattering diagrams in the acyclic finite-type case using Cambrian fans \cite{camb_fan}.
See Theorem~\ref{camb scat thm}.
As explained in Remark~\ref{g method}, the result can be verified by concatenating two results, one that connects cluster scattering diagrams to $\g$-vector fans and one that connects Cambrian fans to $\g$-vector fans, but here we verify it directly using the combinatorics of sortable elements.
We also show, in Corollary~\ref{shard scat}, how to use shards~\cite{shardint} to make a cluster scattering diagram with exactly one wall in each reflecting hyperplane.
For a representation-theoretic approach in the skew-symmetric case, see \cite[Example~10.3]{GHKK}.

As noted above, cluster scattering diagrams of rank $2$ and cluster scattering diagrams of acyclic finite type have been constructed previously in the special case where $B$ is skew-symmetric.
However, even in the skew-symmetric case, our methods of proof are new (or in the case of the Cambrian constructions, have not previously been applied to scattering diagrams). 
We view the construction of cluster scattering diagrams as a combinatorial problem about root systems.
Our aim has been to solve the problem using only the combinatorics of root systems and Coxeter groups and the most basic cluster algebra recursions.

In the cluster algebras literature, there are two different conventions on how to ``extend'' $B$ to add ``coefficients'' to the cluster algebra:  Either by adjoining extra rows to make a ``tall'' matrix or adjoining extra columns to make a ``wide'' matrix.
The difference amounts to replacing $B$ by its transpose $B^T$.
The scattering diagram constructions in \cite{GHKK} fit into the wide-matrix convention, while one of the foundational cluster algebras papers \cite{ca4} uses tall matrices.
An exposition of scattering diagrams in the wide-matrix convention, taking substantially the same point of view as the present paper, is available in \cite{scatfan}.
Here, we rework the definition of scattering diagrams in the tall-matrix convention, at the same time further specializing to principal coefficients.
This allows us to relate scattering diagram results directly to cluster algebra results and constructions from \cite{ca4}, and also serves as a fairly self-contained account of scattering diagrams in the tall-matrix setting.
Because both conventions are prevalent and useful, our terminology and notation consistently identifies the tall-matrix scattering diagrams as ``transposed'' scattering diagrams.
(Replacing $B$ by $-B^T$ corresponds to passing to the Langlands dual in the sense of \cite[Section~1.2]{FG} or \cite[Appendix~A]{GHKK}.
The difference in sign between $B^T$ and $-B^T$ is not very consequential---see Proposition~\ref{scat antip}---so in essence we are applying the constructions of~\cite{GHKK} to the Langlands dual seed.)


\section{Transposed cluster scattering diagrams with principal coefficients}\label{transp sec}
In this section, we begin with an exchange matrix and construct a cluster scattering diagram with principal coefficients.
We introduce a global transpose in order to make a cluster monomial $\thet_{m_0}$ have $\g$-vector $m_0$, in the tall-extended-exchange-matrix sense of \cite{ca4} (as discussed in the Introduction).
A useful side effect of the transpose is that in acyclic finite type the scattering fan for an exchange matrix coincides with the Cambrian fan.
The connection to the Cambrian fan will be made in Section~\ref{camb sec}.
We also use root and weight lattices as part of our initial data.
Some motivation for this choice is provided by a theorem of \cite{Bridgeland,GHKK}, quoted below as Theorem~\ref{root thm}, and by the Cambrian fan construction in acyclic finite type.
Except for the transpose and the language of root and weight lattices, our implementation of principal coefficients follows \cite[Appendix~B]{GHKK}.

\subsection{Exchange matrices, root systems and Coxeter groups}\label{back sec}
We start with an \newword{exchange matrix} $B=[b_{ij}]$, a square integer matrix indexed by $\set{1,\ldots,n}$ that is \newword{skew-symmetrizable}, meaning that there exist real numbers $\delta_i$ such that $\delta_ib_{ij}=-\delta_jb_{ji}$ for all $i,j\in\set{1,\ldots,n}$.
We choose the $\delta_i$ so that $\delta_i^{-1}$ is an integer for each $i$ and $\gcd(\delta_i^{-1}:i\in\set{1,\ldots,n})=1$.
We can do this because $B$ is an integer matrix.
We follow the usual convention and call $n$ the \newword{rank} of $B$.
This conflicts with the usual definition of rank, but not badly:
The constructions considered here require that $B$ be extended, by adjoining rows or columns as discussed above, to obtain a matrix of rank $n$ in the usual sense.

Let $A$ be the \newword{Cartan matrix} associated to $B$.
That is, $A=[a_{ij}]$ where $a_{ii}=2$ for $i=1,\ldots,n$ and $a_{ij}=-|b_{ij}|$ for all distinct $i,j\in\set{1,\ldots,n}$.
In particular, $A$ is \newword{symmetrizable} because $\delta_ia_{ij}=\delta_ja_{ji}$ for all $i,j\in\set{1,\ldots,n}$.

Choose a real vector space $V$ with a distinguished basis $\alpha_1,\ldots,\alpha_n$ called the \newword{simple roots}.
The lattice $Q=\Span_\integers(\alpha_1,\ldots,\alpha_n)$ is called the \newword{root lattice}.
Define the \newword{simple co-roots} to be $\alpha_i\ck=\delta_i^{-1}\alpha_i$, so that $\alpha_1\ck,\ldots,\alpha_n\ck$ is another basis for $V$.
The lattice $Q\ck=\Span_\integers(\alpha_1\ck,\ldots,\alpha_n\ck)$ is the \newword{co-root lattice}.
Since the $\delta_i$ were chosen so that each $\delta_i^{-1}$ is an integer, $Q\ck$ is a sublattice of $Q$ of finite index.

Given a primitive vector $\beta$ in $Q$ (an element $\beta\in Q$ not equal to $k\beta'$ for $k>1$ and $\beta'\in Q$), write $\beta\ck$ for the primitive vector in $Q\ck$ that is a positive scaling of $\beta$.
Given primitive $\beta\ck\in Q\ck$, write $\beta$ for the corresponding primitive vector in $Q$.

Let $K:V\times V\to\reals$ be the bilinear form defined, in the basis of simple roots on the right and simple co-roots on the left, by $K(\alpha_i\ck,\alpha_j)=a_{ij}$.
This restricts to an integer-valued form $K:Q\ck\times Q\to\integers$.
The form is symmetric because 
\[K(\alpha_i,\alpha_j)=\delta_iK(\alpha_i\ck,\alpha_j)=\delta_ia_{ij}=\delta_ja_{ji}=\delta_jK(\alpha_j\ck,\alpha_i)=K(\alpha_j,\alpha_i).\]
Let $\omega:V\times V\to\reals$ be the bilinear form defined by $\omega(\alpha_i\ck,\alpha_j)=b_{ij}$.
This takes integer values on $Q\ck\times Q$ and is skew-symmetric by a similar calculation.

Let $V^*$ be the dual vector space to $V$ and let $\br{\,\cdot\,,\,\cdot\,}:V^*\times V\to\reals$ be the usual pairing.
Define the \newword{fundamental weights} to be the basis $\rho_1,\ldots,\rho_n$ for $V^*$ that is dual to $\alpha_1\ck,\ldots,\alpha_n\ck$, in the sense that $[\br{\rho_i,\alpha\ck_j}]$ is the identity matrix.
(The fundamental weights are dual to the simple \emph{co}-roots, not the simple roots.)
The lattice $P=\Span_\integers(\rho_1,\ldots,\rho_n)$ is called the \newword{weight lattice}.
Define the \newword{fundamental co-weights} to be the basis $\rho_1\ck,\ldots,\rho_n\ck$ for $V^*$ that is dual to $\alpha_1,\ldots,\alpha_n$.
We have $\rho_i=\delta_i\rho_i\ck$ for all $i$.
The lattice $P\ck=\Span_\integers(\rho_1\ck,\ldots,\rho_n\ck)$ is the \newword{co-weight lattice}.
Since each $\delta_i^{-1}$ is an integer, $P$ is a superlattice of $P\ck$ of finite index.

The \newword{dominant chamber} in $V^*$ is the full-dimensional simplicial cone
\begin{equation}\label{D def}
D=\bigcap_{i=1}^n \set{p\in V^*: \br{p,\alpha_i}\ge 0}.
\end{equation}
Equivalently, $D$ is the nonnegative real span of the fundamental weights or of the fundamental co-weights.

For each $i=1,\ldots,n$, define $s_i$ to be the reflection on $V$ given by $s_i(\alpha_j)=\alpha_j-a_{ij}\alpha_i$.
The set $S$ of \newword{simple reflections} $s_i$ generates a \newword{Coxeter group} $W$.
(More precisely, $(W,S)$ is a Coxeter system.)
The action of $s_i$ on $V$ defines an action on $V^*$ in the usual way.
Namely, $s_i$ sends $\lambda\in V^*$ to the unique vector $s_i\lambda\in V^*$ such that $\br{s_i\lambda,s_i\beta}=\br{\lambda,\beta}$ for all $\beta\in V$.

\begin{figure}
\scalebox{0.81}{
\begin{tabular}{ccl}
\\[-7pt]
\raisebox{30pt}{\begin{tabular}{c}
$\begin{bsmallmatrix}2&0\\0&2\end{bsmallmatrix}$\,\,\,\\
$A_1\times A_1$
\end{tabular}}&\hspace{-10pt}
\scalebox{0.93}{\begin{picture}(0,0)(-40,-40)
                \put(47,-1){$\alpha_1$}
                \put(-1,45){$\alpha_2$}
        \end{picture}
        \includegraphics{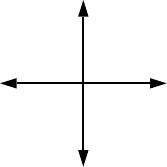}}
&
\scalebox{0.93}{\begin{picture}(0,0)(-20,-40)
                \put(49,-1.5){\textcolor{red}{$\alpha\ck_1$} $=\alpha_1$}
                \put(-15.5,45){\textcolor{red}{$\alpha_2\ck$} $=\alpha_2$}
        \end{picture}
        \includegraphics{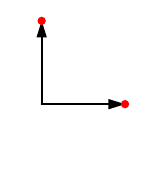}}
\\[0pt]
\raisebox{30pt}{\begin{tabular}{c}
$\begin{bsmallmatrix}\,\,\,\,2&-1\\-1&\,\,\,\,2\end{bsmallmatrix}$\\
\,\,\,$A_2$
\end{tabular}}&\hspace{-10pt}
\scalebox{0.93}{\raisebox{5pt}{\begin{picture}(0,0)(-40,-35)
                \put(46,-1){$\alpha_1$}
                \put(-30,30){$\alpha_2$}
                \put(26,30){$\alpha_1+\alpha_2$}
        \end{picture}
\includegraphics{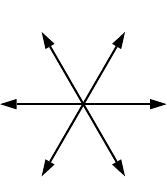}}}
&
\scalebox{0.93}{\raisebox{5pt}{\begin{picture}(0,0)(-20,-40)
                \put(49,-1.5){\textcolor{red}{$\alpha\ck_1$} $=\alpha_1$}
                \put(-15.5,45){\textcolor{red}{$\alpha_2\ck$} $=\alpha_2$}
                \put(28,41){\rotatebox{45}{\textcolor{red}{$\alpha_1\ck+\alpha_2\ck$}}}
                \put(42,30){\rotatebox{45}{$=\alpha_1+\alpha_2$}}
        \end{picture}
\includegraphics{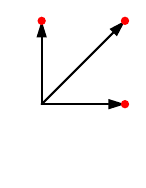}}}
\\[0pt]
\raisebox{30pt}{\begin{tabular}{c}
$\begin{bsmallmatrix}\,\,\,\,2&-1\\-2&\,\,\,\,2\end{bsmallmatrix}$\\
\,\,\,$B_2$
\end{tabular}}&\hspace{-10pt}
\scalebox{0.93}{\begin{picture}(0,0)(-40,-40)
                \put(47,-1){$\alpha_1$}
                \put(-31,20){$\alpha_2$}
                \put(28,20){$\alpha_1+\alpha_2$}
                \put(8,36){$\alpha_1+2\alpha_2$}
        \end{picture}
\includegraphics{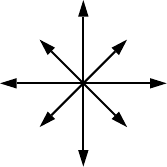}}
&
\scalebox{0.93}{\begin{picture}(0,0)(-20,-40)
                \put(50,-1.5){\textcolor{red}{$\alpha_1\ck$} $=\alpha_1$}
                \put(-11,18){$\alpha_2$}
                \put(-11,39){\textcolor{red}{$\alpha_2\ck$}}
                \put(28,41){\rotatebox{45}{\textcolor{red}{$\alpha_1\ck+\alpha_2\ck$}}}
                \put(42,30){\rotatebox{45}{$=\alpha_1+2\alpha_2$}}
               \put(86,41){\rotatebox{26.6}{\textcolor{red}{$2\alpha_1\ck+\alpha_2\ck$}}}
                \put(40,11){\rotatebox{26.6}{$\alpha_1+\alpha_2$}}
        \end{picture}
\includegraphics{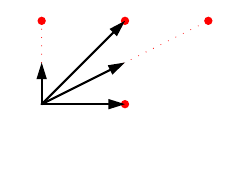}}
\\[-15pt]
\raisebox{50pt}{\begin{tabular}{c}
$\begin{bsmallmatrix}\,\,\,\,2&-1\\-3&\,\,\,\,2\end{bsmallmatrix}$\\
\,\,\,$G_2$
\end{tabular}}&\hspace{-10pt}
\scalebox{0.93}{\begin{picture}(0,0)(-76,-65)
                \put(80,-2){$\alpha_1$}
                \put(-46,22){$\alpha_2$}
                \put(44,22){$\alpha_1+\alpha_2$}
                \put(-75,62){$\alpha_1+3\alpha_2$}
                \put(-13,47){$\alpha_1+2\alpha_2$}
                \put(44,62){$2\alpha_1+3\alpha_2$}
        \end{picture}
\includegraphics{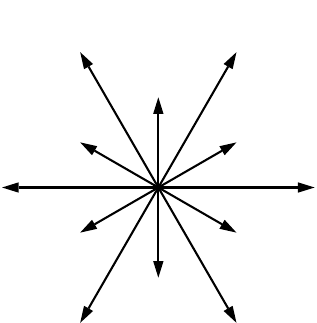}\hspace{0pt}}
&\hspace{0pt}
\scalebox{0.93}{\begin{picture}(0,0)(-20,-65)
                \put(50,-1.5){\textcolor{red}{$\alpha_1\ck$} $=\alpha_1$}
                \put(-11,18){$\alpha_2$}
                \put(-11,39){\textcolor{red}{$\alpha_2\ck$}}
                \put(28,41){\rotatebox{45}{\textcolor{red}{$\alpha_1\ck+\alpha_2\ck$}}}
                \put(45,43){\rotatebox{45}{$=\alpha_1+3\alpha_2$}}
               \put(83,43){\rotatebox{26.6}{\textcolor{red}{$2\alpha_1\ck+\alpha_2\ck$}}}
               \put(83,31){\rotatebox{26.6}{$=2\alpha_1+3\alpha_2$}}
                \put(40,5){\rotatebox{18.4}{$\alpha_1+\alpha_2$}}
               \put(40,29){\rotatebox{33.7}{$\alpha_1+2\alpha_2$}}
                \put(95,20){\rotatebox{18.4}{\textcolor{red}{$3\alpha_1\ck+\alpha_2\ck$}}}
               \put(95,70){\rotatebox{33.7}{\textcolor{red}{$3\alpha_1\ck+2\alpha_2\ck$}}}
        \end{picture}
\includegraphics{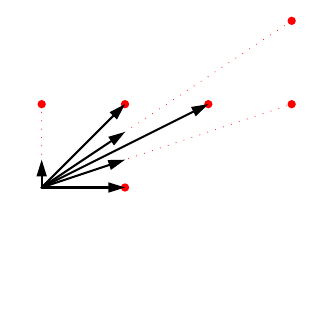}\hspace{0pt}}
\end{tabular}
}
\caption{Finite crystallographic root systems with $n=2$}
\label{rk2roots}
\end{figure}
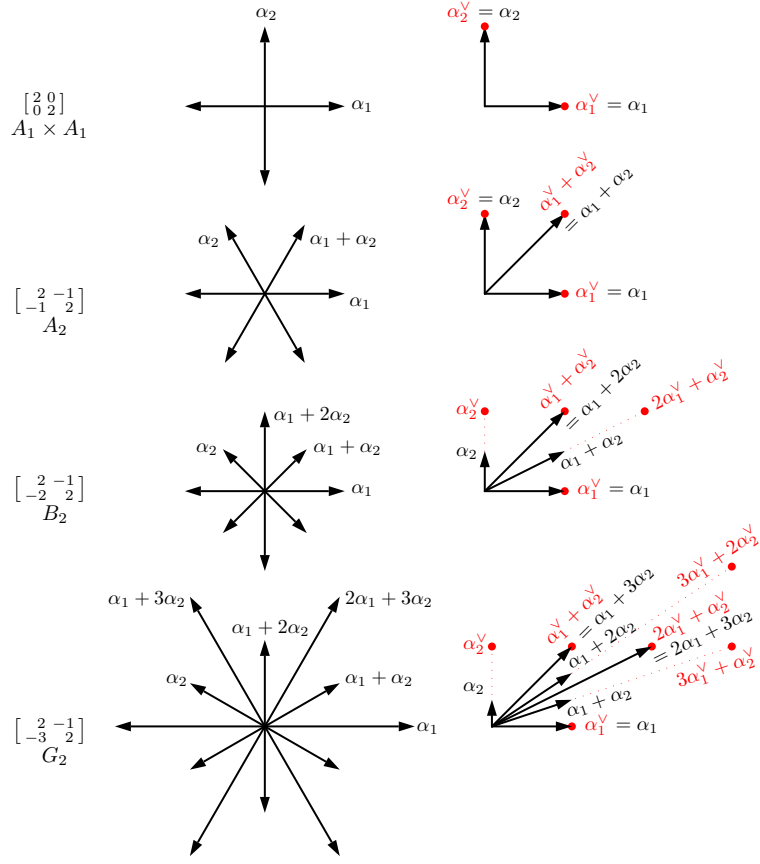

The \newword{real roots} $\Phi^\re$ associated to $A$ are the vectors of the form $w\alpha_i$ for $w\in W$ and $i=1,\ldots,n$.
When $\Phi^\re$ is infinite, there are also \newword{imaginary} roots $\Phi^\im$ associated to $A$, which we need not define here.
The \newword{root system} associated to $A$ is the disjoint union $\Phi=\Phi^\re\cup\Phi^\im$.
(Root systems as we have defined them are sometimes called \newword{crystallographic}; there is a more general notion that does not concern us here.)
The root system $\Phi$ is also a disjoint union $\Phi=\Phi_+\cup\Phi_-$ such that each $\beta\in\Phi_+$ is a nonnegative linear combination of simple roots and each $\beta\in\Phi_-$ is a nonpositive linear combination of simple roots.
The roots in $\Phi_+$ are called the \newword{positive roots}.
We write $Q^+$ for the subset of $Q$ consisting of nonzero vectors obtained as nonnegative integer combinations of simple roots. 

The \newword{real co-roots} 
are the vectors of the form $w\alpha\ck_i$ for $w\in W$ and $i=1,\ldots,n$.
The vector $\beta\ck$ is a real co-root if and only if its scaling $\beta$ is a real root.

Figure~\ref{rk2roots} shows the finite (crystallographic) root systems for $n=2$, with their Cartan matrices and types.
The pictures in the middle column are drawn so that the form $K$ agrees with the usual Euclidean metric on the plane of the page.
The column on the right shows the positive roots and co-roots in a less conventional way:
The simple co-roots point right and up (and have the same length in the sense of the page).
Co-roots are shown as red dots, while roots are shown as arrows.

\begin{remark}\label{my way is better}
We have placed roots and co-roots in the same vector space and placed weights and co-weights in the dual space.
It is common (for example in the theory of Kac-Moody Lie algebras as in \cite{Kac}) to place roots and weights in the same vector space, place co-roots and co-weights in the dual space, and let the natural pairing play the role that we have given to $K$.
The approach here agrees with our approach in earlier papers, including \cite{typefree,framework,afframe,cyclica,affdenom} and eliminates the need to enlarge the vector spaces. 
Most importantly, the present approach lines up perfectly with the definition of scattering diagrams in \cite{GHKK}.
\end{remark}

\subsection{Principal-coefficients transposed scattering diagrams in root notation}\label{prin trans sec}
Table~\ref{scat root data} describes the initial data for a transposed scattering diagram with principal coefficients.
The only input is an exchange matrix $B=[b_{ij}]$, from which we extract a Cartan matrix and make the definitions of Section~\ref{back sec}.
\begin{longtable}{P{70pt}|P{260pt}}
\caption{Initial data for a transposed scattering diagram with principal coefficients}\label{scat root data}\\
Notation&Description/requirements\\\hline\hline
\endfirsthead
\caption{(continued)}\\
Notation&Description/requirements\\\hline\hline
\endhead
$Q^+$&$\set{\sum_{i=1,\ldots,n}a_i\alpha_i:a_i\in\integers,\,a_i\ge0,\,\sum_{i=1,\ldots,n}a_i>0}$\newline positive part of root lattice \\\hline
$x_1,\ldots,x_n$\newline $y_1,\ldots,y_n$&indeterminates\\\hline
$x^\lambda y^\beta$&$x_1^{a_1}\cdots x_n^{a_n}y_1^{c_1}\cdots y_n^{c_n}$ for $(\lambda,\beta)=(\sum_{i=1}^na_i\rho_i,\sum_{i=1}^nc_i\alpha_i)\in P\oplus Q$\\\hline
$\hy_1,\ldots,\hy_n$&$\hy_i=y_ix_1^{b_{1i}}\cdots x_n^{b_{ni}}$\\\hline
$x^\lambda \hy^\beta$&
$x_1^{a_1}\cdots x_n^{a_n}\hy_1^{c_1}\cdots\hy_n^{c_n}$ for $(\lambda,\beta)=(\sum_{i=1}^na_i\rho_i,\sum_{i=1}^nc_i\alpha_i)\in P\oplus Q$\\\hline
$\k$&a field of characteristic zero\\\hline
$\k[[\hy]]$&$\k[[\hy_1,\ldots,\hy_n]]$\\\hline
$\k[[x,\hy]]$&$\k[[x_1,\ldots,x_n,\hy_1,\ldots,\hy_n]]$\\\hline
$\m$&ideal in $\k[[x,\hy]]$ consisting of series with constant term zero\\\hline
\end{longtable}
As discussed in the Introduction, we work with a global transpose, relative to~\cite{GHKK}.
To avoid confusion, we will be explicit about this transpose in terminology and notation.
We also follow \cite{scatfan} in working in a lower-dimensional space than \cite{GHKK}.
For details, including an explanation of why the additional dimensions are unnecessary, see \cite[Remarks~2.1, 2.12,~2.13]{scatfan}.

For the purpose of comparison, Table~\ref{scat root dict} gives this initial data in the context of the more general setup of \cite{GHKK}.
Table~\ref{scat root dict} is designed for easy comparison with \cite[Table~1]{scatfan}.  
The general setup leaves several choices, and we have made these choices in ways that are natural to the root-system context. 
We see the explicit transpose in the line of Table~\ref{scat root dict} describing $\epsilon_{ij}$.
\begin{longtable}{P{88pt}|P{242pt}}
\caption{Initial data for a transposed scattering diagram with principal coefficients, in the context of the general setup}\label{scat root dict}\\
General setup&Transposed, principal coefficients\\\hline\hline
\endfirsthead
\caption{(continued)}\\
General setup&Transposed, principal coefficients\\\hline\hline
\endhead
$N$&$Q\oplus P$\quad (root lattice and weight lattice)\\\hline
$M=\Hom(N,\integers)$&$P\ck\oplus Q\ck$\quad (co-weight lattice and co-root lattice)\\\hline
$N_\uf\subseteq N$&Root lattice $Q$, identified with $Q\oplus\set{0}\subset Q\oplus P$\\\hline
$I$&Two copies of $\set{1,\ldots,n}$, one indexing the simple roots in $N$ and one indexing the fundamental weights in $N$\\\hline
$I_\uf\subseteq I$&$\set{1,\ldots,n}$, indexing the simple roots in $N$\\\hline
$\set{\,\cdot\,,\,\cdot\,}:N\times N\to\rationals$&$\set{(\beta_1,\lambda_1),(\beta_2,\lambda_2)}=-\omega(\beta_1,\beta_2)+\br{\lambda_2,\beta_1}-\br{\lambda_1,\beta_2}$\\\hline
$N^\circ\subseteq N$&$Q\ck\oplus P\ck$\quad (co-root lattice and co-weight lattice)\\\hline
$M^\circ=\Hom(N^\circ,\integers)$&$P\oplus Q$\quad (weight lattice and root lattice)\\\hline
$d_i$&$d_i=\delta_i^{-1}=\frac{K(\alpha_i,\alpha_i)}2=\frac2{K(\alpha_i\ck,\alpha_i\ck)}$ \quad so $\alpha_i\ck=d_i\alpha_i$\newline
$d_i$ same on both copies of $\set{1,\ldots,n}$, so $\rho_i\ck=d_i\rho_i$\\\hline
$\s=(e_i:i\in I)$&$\s=(\alpha_1,\ldots,\alpha_n,\rho_1,\ldots,\rho_n)$\quad basis for $Q\oplus P$\\\hline
$N^+=N^+_\s$&$Q^+=\set{\sum_{i=1}^na_i\alpha_i:a_i\in\integers,\,a_i\ge0,\,\sum_{i=1}^na_i>0}$\newline positive part of root lattice \\\hline
$[\,\cdot\,,\,\cdot\,]_\s:N\times N\to\rationals$&
$[\alpha_i,\alpha_j]_\s=\set{\alpha_i,\alpha_j\ck}=\omega(\alpha_j\ck,\alpha_i)=b_{ji}$\newline
$[\alpha_i,\rho_j]_\s=\br{\rho_j\ck,\alpha_i}=\delta_{ij}$ (Kronecker delta)\newline
$[\rho_i,\alpha_j]_\s=-\br{\rho_i,\alpha_j\ck}=-\delta_{ij}$ \newline
$[\rho_i,\rho_j]_\s=0$
\\\hline
$\epsilon_{ij}=[e_i,e_j]_\s$&entries in matrix \raisebox{0pt}[14pt][9pt]{$\begin{bsmallmatrix}\,B^T&\,I\\-I&0\end{bsmallmatrix}$} ($n\times n$ blocks)\\\hline
$(e^*_i:i\in I)$&$(\rho_i\ck,\ldots,\rho_n\ck,\alpha_1\ck,\ldots,\alpha_n\ck)$\quad basis for $P\ck\oplus Q\ck$\\\hline
$(f_i:i\in I)$&$(\rho_i,\ldots,\rho_n,\alpha_1,\ldots,\alpha_n)$\quad basis for $P\oplus Q$\\\hline
$p^*:N_\uf\to M^\circ$&$p^*(\alpha_i)=(\sum_{j=1}^nb_{ji}\rho_j,\,\alpha_i)$ 
\newline
$p^*(\beta)=(\phi\ck\ni Q\ck\mapsto\omega(\phi\ck,\beta),\beta)$ for $\beta\in Q$
\\\hline
$(v_i\in M^\circ:i\in I)$&$(p^*(\alpha_1),\ldots,p^*(\alpha_n),\rho_1,\ldots,\rho_n)$\\\hline
$(z_i:i\in I)$&$(x_1,\ldots,x_n,y_1,\ldots,y_n)$\quad indeterminates\\\hline
$z^{(\lambda,\beta)}$&$x^\lambda y^\beta=x_1^{a_1}\cdots x_n^{a_n}y_1^{c_1}\cdots y_n^{c_n}$ for $(\lambda,\beta)=(\sum_{i=1}^na_i\rho_i,\sum_{i=1}^nc_i\alpha_i)\in P\oplus Q$\\\hline
$(\zeta_i=z^{v_i}:i\in I)$&$(\hy_1,\ldots,\hy_n,x_1,\ldots,x_n)$ for $\hy_i=y_ix_1^{b_{1i}}\cdots x_n^{b_{ni}}$\\\hline
$\zeta^{(\beta,\lambda)}$&
$x^\lambda \hy^\beta=x_1^{a_1}\cdots x_n^{a_n}\hy_1^{c_1}\cdots\hy_n^{c_n}$ for $(\beta,\lambda)=(\sum_{i=1}^nc_i\alpha_i,\sum_{i=1}^na_i\rho_i)\in Q\oplus P$\\\hline
$\k$&a field of characteristic zero\\\hline
$\k[[\zeta]]$&$\k[[x,\hy]]=\k[[x_1,\ldots,x_n,\hy_1,\ldots,\hy_n]]$\\\hline
$\m$&ideal in $\k[[x,\hy]]$ consisting of series with constant term zero\\\hline
\end{longtable}

We work in the formal power series ring $\k[[x,\hy]]$ or, sometimes for convenience, in the quotient $\k[[x,\hy]]/\m^{k+1}$ for $k\ge 0$.
A \newword{wall} $(\d,f_\d)$ consists of a codimension-$1$ cone $\d$ in $V^*$ and a function $f_\d\in\k[[\hy]]$ (or in $\k[[\hy]]/\m^{k+1}$) such that:
\begin{enumerate}[(i)]
\item $\d$ is contained in $\beta^\perp$ for some primitive $\beta\in Q^+$ and defined by inequalities of the form $\br{p,\phi}\le0$ for $\phi\in Q$.
\item $f_\d=f_\d(\hy^\beta)$ is in the univariate power series ring $\k[[\hy^\beta]]$ for this primitive $\beta$ (or $f_\d\in\k[[\hy^\beta]]/\br{\hy^{(k+1)\beta}}$).
\end{enumerate}
A wall $(\d,f_\d)$ is \newword{incoming} if the vector $\omega(\,\cdot\,,\beta)\in V^*$ is in $\d$ and otherwise it is \newword{outgoing}.
Two walls are \newword{parallel} if they are contained in the same hyperplane.

A \newword{scattering diagram} is a collection $\D$ of walls such that the set $\D_k$ of walls $(\d,f_\d)\in\D$ with $f_\d\not\equiv 1$ modulo $\m^{k+1}$ is finite for all $k\ge 1$.
The \newword{support} $\Supp(\D)$ is the union of the walls of $\D$.

Given a scattering diagram $\D$, a \newword{generic} path for $\D$ is a piecewise differentiable path $\gamma:[0,1]\to V^*$ that:
\begin{itemize}
\item does not pass through the intersection of any two non-parallel walls of $\D$;
\item does not pass through the relative boundary of any wall;
\item has endpoints $\gamma(0)$ and $\gamma(1)$ contained in $V^*\setminus\Supp(\D)$; and
\item crosses walls only transversely.
\end{itemize}

Suppose $\gamma$ is a generic path for $\D$ and $(\d,f_\d)$ is a wall of $\D$ with $\gamma(t)\in\d$ for some $t\in(0,1)$.
The \newword{wall-crossing automorphism} $\p_{\gamma,\d,t}:\k[[x,\hy]]\to\k[[x,\hat y]]$ associated to this crossing is given by 
\begin{align}
\label{theta def x}
\p_{\gamma,\d,t}(x^\lambda)&=x^\lambda f_\d^{\br{\lambda,\pm\beta\ck}},\\
\label{theta def hat y}
\p_{\gamma,\d,t}(\hy^\phi)&=\hy^\phi f_\d^{\omega(\pm\beta\ck\!,\,\phi)},
\end{align}
where $\beta\ck$ is the normal vector to $\d$ that is contained in $Q^+$ and is primitive in $Q\ck$, taking $+\beta\ck$ if $\br{\gamma'(t),\beta}<0$ or $-\beta\ck$ if $\br{\gamma'(t),\beta}>0$.
(If $\gamma$ is not differentiable at $t$, the sign of $\br{\gamma'(t),\beta}$ still makes sense, recording the direction in which $\gamma$ crosses the wall.)
The explicit dependence on $t$ in the notation $\p_{\gamma,\d,t}$ is meant to emphasize that $\gamma$ might cross $\d$ multiple times, and in different directions, so that the sign chosen in \eqref{theta def x} and \eqref{theta def hat y} depends on more than just $\gamma$ and $\d$.

For each $k\ge 1$, define $\p_{\gamma,\D_k}$ to be $\p_{\gamma,\d_\ell,t_\ell}\circ\cdots\circ\p_{\gamma,\d_1,t_1}:\k[[x,\hy]]\to\k[[x,\hy]]$ such that $\d_1,\ldots,\d_\ell$ is the sequence of walls of $\D_k$ crossed by $\gamma$ with $\d_i$ crossed at time $t_i$ and $t_1\le t_2\le\cdots\le t_\ell$.
(There is a finite sequence of crossings because $\D_k$ is finite and because $\gamma$ is generic.)
Define the \newword{path-ordered product} $\p_{\gamma,\D}:\k[[x,\hy]]\to\k[[x,\hy]]$ to be $\lim_{k\to\infty}\p_{\gamma,\D_k}$.
We say $\D$ is \newword{consistent} if $\p_{\gamma,\D}$ depends only on $\gamma(0)$ and $\gamma(1)$.
By \cite[Proposition~2.4]{scatfan}, $\D$ is consistent if and only if each $\D_k$ is consistent modulo $\m^{k+1}$.
When $\D$ is consistent and $p,q\in V^*\setminus\Supp(\D)$, we define $\p_{p,q,\D}=\p_{\gamma,\D}$ for $\gamma$ a generic path from $p$ to $q$, which exists by \cite[Proposition~2.2]{scatfan}.

It is useful to reinterpret $\p_{\gamma,\D}$ as a map on Laurent monomials sending $x^\lambda y^\phi$ to $x^\lambda y^\phi f_\d^{\br{\lambda,\pm\beta\ck}}$ (with the choice of sign for $\pm\beta\ck$ as in the definition).
With that interpretation, the following is \cite[Proposition~2.5]{scatfan}, translated into our setup.
\begin{prop}\label{just x's}  
A path-ordered product $\p_{\gamma,\D}$ is determined entirely by its values $\p_{\gamma,\D}(x_1),\ldots,\p_{\gamma,\D}(x_n)$, or equivalently by its values $\p_{\gamma,\D}(x^\lambda)$ for $\lambda\in P$.
\end{prop}

Two scattering diagrams $\D$ and $\D'$ are \newword{equivalent} if and only if $\p_{p,q,\D}=\p_{p,q,\D'}$ for all $p,q\in V^*\setminus(\Supp(\D)\cup\Supp(\D'))$.
A \newword{general} point is a point $p\in V^*$ contained in at most one hyperplane $\beta^\perp$ with $\beta\in Q^+$.
Given a scattering diagram $\D$ and a general point $p$, write $f_p(\D)=\prod_{\d\ni p}f_\d\in\k[[\hy]]$.
By \cite[Lemma~1.9]{GHKK}, two scattering diagrams $\D$ and $\D'$ are equivalent if and only $f_p(\D)$ and $f_p(\D')$ agree on all general points $p$.
A scattering diagram has \newword{minimal support} if no equivalent scattering diagram has strictly smaller support.
Every consistent scattering diagram is equivalent to a scattering diagram $\D$ with minimal support and such that each $\D_k$ has minimal support \cite[Proposition~2.8]{scatfan}.

Given a scattering diagram $\D$ and $\beta\in Q^+$, the \newword{rampart} of $\D$ associated to $\beta$ is the union of all supports of walls of $\D$ contained in $\beta^\perp$.  
For $p\in V^*$, write $\Ram_\D(p)$ for the set of ramparts $R$ of $\D$ such that $p\in R$, and write $\D\setminus\Ram_\D(p)$ for the set of walls of $\D$ not contained in any rampart in $\Ram_\D(p)$.

If $\D$ is consistent and has minimal support, we say $p,q\in V^*$ are $\D$-equivalent if and only if there is a path $\gamma$ from $p$ to $q$ on which $\Ram_\D(\,\cdot\,)$ is constant.
This happens if and only if $\Ram_\D(p)=\Ram_\D(q)$ and $p$ and $q$ are in the same path-connected component of $(\cap\Ram_\D(p))\setminus(\Supp(\D\setminus\Ram_\D(p)))$.
The closure of a $\D$-class is a \newword{$\D$-cone}.

A \newword{closed convex cone} in $V^*$ is a subset of $V^*$ that is closed in the usual sense, and also closed under addition and closed under nonnegative scaling.
A subset $F$ of a closed convex cone $C$ is a \newword{face} of $C$ if it is cone and has the property that if $x,y\in C$ and $F$ contains some point in the line segment $\seg{xy}$ besides $x$ and $y$, the whole segment $\seg{xy}$ is in $F$.
A \newword{fan} is a set $\F$ of closed convex cones such that if $C\in\F$ then every face of $C$ is in $\F$, and such that if $C,D\in\F$ then $C\cap D$ is a face of $C$ and a face of $D$.
We write $|\F|$ for the union of the cones in $\F$.
A fan $\F$ is \newword{complete} if $|\F|$ is the entire ambient vector space.

If $\D$ is consistent and has minimal support, then each $\D$-cone is a closed convex cone \cite[Proposition~3.5]{scatfan} and the collection $\Fan(\D)$ of all $\D$-cones and their faces is a complete fan in $V^*$ by \cite[Theorem~3.1]{scatfan}.

\subsection{Transposed cluster scattering diagrams}\label{trans clus sec}
An exchange matrix $B$ determines the \newword{transposed cluster scattering diagram with principal coefficients} $\Scat^T(B)$.
This is the unique (up to equivalence) consistent scattering diagram that is obtained by appending \emph{outgoing} walls to an initial scattering diagram $\set{(\alpha_i^\perp,1+\hy_i):i=1,\ldots,n}$.
In the more general (``non-transposed'') setup of \cite{scatfan} (as described in Table~\ref{scat root dict}), the scattering diagram $\Scat^T(B)$ is $\Scat(\tB,\s)$ for $\tB=[B^T\,\,I_n]$, where $I_n$ is an identity matrix, and $\s=(\alpha_1,\ldots,\alpha_n,\rho_1,\ldots,\rho_n)$.
We write $\Scat^T(B)$ rather than $\Scat(B^T)$ because we think of $B$ as our primary object, and we want to think of the global transpose not as a modification of $B$, but as a choice of conventions for creating a scattering diagram.
The fan $\Fan(\Scat^T(B))$ is denoted $\ScatFan^T(B)$ and called the \newword{transposed scattering fan}.

We pause here to quote a result that supports the use of root systems in our setup for scattering diagrams.

\begin{theorem}\label{root thm}
If $B$ is skew-symmetric and acyclic, then every wall of $\Scat^T(B)$ is normal to a root.
\end{theorem}

\textit{A priori}, every wall is normal to a positive vector in the root lattice, not necessarily a root.
Theorem~\ref{root thm} follows from \cite[Proposition~10.1]{GHKK}.
(See also \cite[Example~10.4]{GHKK}.) 
It also follows from \cite[Lemma~11.4]{Bridgeland} when $B$ is non-degenerate, relaxing the acyclicity requirement but requiring the existence of a genteel potential. 
We expect that the theorem extends to the case where $B$ is merely skew-symmetrizable.

The following useful fact about transposed cluster scattering diagrams with principal coefficients is proved by applying the antipodal map throughout and observing that the sign changes cancel when we check consistency.
\begin{prop}\label{scat antip}
For any exchange matrix $B$, 
\[\Scat^T(-B)=\set{(-\d,f_\d((\hy')^\beta)):\,(\d,f_\d(\hy^\beta)\in\Scat^T(B)},\]
where the $\hy_i$ are the monomials defined as in Table~\ref{scat root data} using $B$ while the $\hy'_i$ are defined using $-B$, and $f_\d((\hy')^\beta)$ is obtained from $f_\d(\hy^\beta)$ by replacing each $\hy_i$ by $\hy_i'$.
\end{prop}

We write $\A_\bullet(B)$ for the cluster algebra with principal coefficients associated to $B$, in the sense of \cite[Definition~3.1]{ca4}.
As discussed above, we have passed from \emph{wide} to \emph{tall} extended exchange matrices, and this is the reason for dealing with \emph{transposed} scattering diagrams.
Thus $\A_\bullet(B)$ is the cluster algebra associated (as in \cite[Definition~2.12]{ca4}) to the extended exchange matrix $\begin{bsmallmatrix}B\\I\end{bsmallmatrix}$, where $I$ is an identity matrix.

A \newword{cluster monomial} is a monomial in the cluster variables in some seed of $\A_\bullet(B)$.
(Conventions on cluster monomials differ, but we take monomials in the ``unfrozen'' cluster variables, not including the ``frozen''/tropical variables.)
We quote two constructions of cluster monomials for $\A_\bullet(B)$.
The first is in terms of broken lines.

Fix a point $p$ in the dominant chamber $D$ such that $p$ is not contained in any hyperplane $\beta^\perp$ for $\beta\in Q$.
We will define a \newword{theta function} $\thet_\lambda$ for every nonzero weight $\lambda\in P\setminus\set{0}$.
Our definition will not depend on the choice of $p$.
(This is not obvious, but rather follows from \cite[Theorem~3.5]{GHKK}, which is a special case of results of \cite[Section~4]{CPS}.)

Let $\gamma:(-\infty,0]\to V^*$ be a piecewise linear path with finitely many domains of linearity.
To each domain $L$ of linearity of $\gamma$, assign a monomial $c_Lx^{\lambda_L}y^{\beta_L}$ with $c_L\in\k$ and $(\lambda_L,\beta_L)\in P\oplus Q$.
Then $\gamma$ is a \newword{broken line} for $\lambda$ with endpoint $p$ if the path and the monomials satisfy the following conditions.  
\begin{enumerate}[(i)]
\item $\gamma(0)=p$.
\item \label{brok generic root}
$\gamma$ is disjoint from all relative boundaries of walls of $\Scat^T(B)$ and disjoint from all intersections of non-parallel walls of $\Scat^T(B)$. 
\item In each domain $L$ of linearity, $\gamma'$ is constantly equal to $-\lambda_L$.
\item If $L$ is the unbounded domain of linearity of $\gamma$, then $c_Lx^{\lambda_L}y^{\beta_L}=x^\lambda$.
\item \label{change slope}
At each point $t$ of nonlinearity, passing (as the parameter increases) from a domain $L$ of linearity to a domain $L'$ of linearity, by \eqref{brok generic root} there exists $\beta\ck$ primitive in $Q\ck$ such that all walls containing $\gamma(t)$ are in $(\beta\ck)^\perp$ and ${\br{\lambda_L,\beta\ck}>0}$.
If $f$ is the product of the $f_\d$ for all walls $(\d,f_\d)$ with $\gamma(t)\in\d$, then $c_{L'}x^{\lambda_{L'}}y^{\beta_{L'}}$ equals $c_Lx^{\lambda_L}y^{\beta_L}$ times a term in $f^{\br{\lambda_L,\beta\ck}}$.
\end{enumerate}
These conditions in particular allow us to recover the monomials from the path $\gamma$.

Writing $c_\gamma x^{\lambda_\gamma}y^{\beta_\gamma}$ for the monomial on the domain of linearity containing $0$, we define the theta function $\thet_\lambda$ to be the sum, over all broken lines for $\lambda$ with endpoint $p$, of the monomials $c_\gamma x^{\lambda_\gamma}y^{\beta_\gamma}$.
This is an element of $x^\lambda \k[[\hy]]$.
(We emphasize that the monomial on a domain $L$ of linearity is $c_Lx^{\lambda_L}y^{\beta_L}$, not $c_Lx^{\lambda_L}\hy^{\beta_L}$.
Thus when the monomial changes as described in \eqref{change slope}, both $\lambda_L$ and $\beta_L$ change.)

The subtleties inherent in computing theta functions are compounded by the appearance of both roots and co-roots in the definition.
We give some examples of computing theta functions in rank $2$ in Section~\ref{rk2 theta sec}.

The dominant chamber $D$ is a cone in $\ScatFan^T(B)$.  
Write $\ChamberFan^T(B)$ for the subfan of $\ScatFan^T(B)$ consisting of $D$, all maximal cones $D'$ adjacent to $D$, all maximal cones adjacent to such $D'$, etc., together with all faces of these cones.
The notation $\ChamberFan^T(B)$ will be short-lived in this paper, as almost immediately it will be replaced by a more enlightening notation.
The following is immediate from \cite[Theorem~5.2]{scatfan} (a version of \cite[Theorem~4.9]{GHKK}), from \cite[Corollary~5.9]{GHKK}, and from the definition of $\g$-vectors in \cite[Section~6]{ca4}.

\begin{theorem}\label{clus mon thm}
The map $\lambda\mapsto\thet_\lambda$ is a bijection from $P\cap|\ChamberFan^T(B)|$ to the set of cluster monomials in $\A_\bullet(B)$.
If $\lambda\in P\cap|\ChamberFan^T(B)|$, then $\lambda$ is the $\g$-vector of the cluster monomial $\thet_\lambda$.
There is a bijection from rays of $\ChamberFan^T(B)$ to cluster variables in $\A_\bullet(B)$ sending each ray to $\thet_\lambda$, where $\lambda$ is the shortest vector in $P$ contained in the ray.
\end{theorem}

The $\g$-vector is defined to be an integer vector, but we interpret elements of~$P$ as $\g$-vectors by taking fundamental-weight coordinates.
Define $\gFan(B)$ to be the set of all cones $C$ such that $C$ is the nonnegative linear span of the $\g$-vectors of a subset of a cluster of $\A_\bullet(B)$.
The following dual version of \cite[Theorem~0.8]{GHKK} is an immediate corollary of Theorem~\ref{clus mon thm}.

\begin{corollary}\label{cham g}
The set $\gFan(B)$ is a fan and coincides with $\ChamberFan^T(B)$.
\end{corollary}

Accordingly, we call $\gFan(B)$ the \newword{$\g$-vector fan} of $B$, and we will refer to $\gFan(B)$ rather than $\ChamberFan^T(B)$ through the rest of the paper.

\begin{remark}\label{Clear fr}
In \cite[Theorem~5.2]{scatfan}, there is an operator $\Clear_\fr$ that does not appear in Theorem~\ref{clus mon thm} because the latter concerns principal coefficients.
In the language of \cite{scatfan}, this is because each term of each $\thet_\lambda$ contains only positive powers of the $\hy_i$, each of which contains only positive powers of the frozen variables~$y_j$.
\end{remark}

The following is \cite[Theorem~4.6]{scatfan} in our transposed principal-coefficients setting.
\begin{theorem}\label{clus easy root}
If $F$ and $G$ are adjacent maximal cones of $\gFan(B)$, then the function $f_p(\Scat^T(B))$ is $1+\hy^{\beta}$ for every general point $p$ in $F\cap G$, where $\beta$ is the primitive normal to $F\cap G$ in $Q^+$.
\end{theorem}

Our second construction of cluster monomials is in terms of path-ordered products.
The following is a rephrasing of \cite[Theorem~5.6]{scatfan}.

\begin{theorem}\label{clus mon pop root}
If $\lambda\in P$ is contained in a cone of $\gFan(B)$, then the cluster monomial $\thet_\lambda$ with $\g$-vector $\lambda$ is $\p_{q,p,\Scat^T(B)}(x^\lambda)$ for any point $p$ in the interior of the dominant chamber $D$ and any point $q$ in the interior of a maximal cone $C$ of $\gFan(B)$ such that $\lambda\in C$.
\end{theorem}

The function $\p_{q,p,\Scat^T(B)}(x^\lambda)$ is $F_\lambda\cdot x^\lambda$ for some $F_\lambda\in\k[[\hy]]$.
By \cite[Corollary~6.3]{ca4}, the cluster variable with $\g$-vector $\lambda$ is $x^\lambda$ times a polynomial in the $\hy$ called the \newword{$F$-polynomial} of the cluster variable.
(This works because in the principal coefficients case, the denominator in \cite[(6.5)]{ca4} is $1$.)
Combining Theorem~\ref{clus mon thm} and Theorem~\ref{clus mon pop root}, we have the following immediate corollary.

\begin{cor}\label{clus mon F}
If $\lambda\in P$ is the $\g$-vector of a cluster variable, the corresponding $F$-polynomial is $x^{-\lambda}\cdot\p_{q,p,\Scat^T(B)}(x^\lambda)$ for any point $p$ in the interior of the dominant chamber $D$ and any point $q$ in the interior of a maximal cone $C$ of $\gFan(B)$ such that $\lambda\in C$.
\end{cor}

We conclude this section by pointing out how the results of \cite{GHKK}, as rephrased in \cite{scatfan}, prove \cite[Conjecture~7.12]{ca4}, one of the major conjectures of \cite{ca4}.
We write $\mu_k$ for matrix mutation in direction $k$ and $\eta_k^B:V^*\to V^*$ for the mutation map as defined, for example, in \cite[Section~4.1]{scatfan}.
In light of Corollary~\ref{cham g}, the following result is a consequence of \cite[Corollary~4.5]{scatfan}, which in turn is a consequence of \cite[Theorem~1.24]{GHKK}.
\begin{theorem}\label{7.12}
For any exchange matrix~$B$ and any $k\in\set{1,\ldots,n}$, the mutation map $\eta^{B^T}_k$ is a piecewise-linear isomorphism from $\gFan(B)$ to $\gFan(\mu_k(B))$.
\end{theorem}

Each cluster monomial in $\A_\bullet(B)$ is also a cluster monomial in $\A_\bullet(\mu_k(B))$.
The conjecture \cite[Conjecture~7.12]{ca4} states that the $\g$-vector of the cluster monomial with respect to $\mu_k(B)$ is obtained from its $\g$-vector with respect to $B$ by a particular piecewise-linear map.
In \cite[Section~8]{universal}, it was pointed out that this map is the mutation map $\eta_k^{B^T}$.
In light of Theorem~\ref{clus mon thm} and the fact that $\eta^{B^T}_k$ is an automorphism of $P$, Theorem~\ref{7.12} is \cite[Conjecture~7.12]{ca4}.

%
%

\section{Scattering diagrams of rank $2$}\label{rk2 sec}

Rank-$2$ scattering diagrams of finite type are easy to understand and are treated as part of the finite-type discussion in Section~\ref{camb sec}.
In contrast, rank-$2$ scattering diagrams of non-finite type are complicated.  
In most cases, there is a region where the walls of the diagram are not well understood.
Combining \cite[Theorem~1.1]{Bridgeland} and \cite[Theorem~1.5]{Bridgeland} yields a formula that in principle describes these walls when $B$ is skew-symmetric, but the details are quite complicated.
In \cite[Example~1.15]{GHKK}, the expectation is expressed that every rational ray in this region is a wall with a nontrivial function attached. 
If this is true, then every ray (rational or not) in the region is a cone in the scattering fan.
When the Cartan matrix associated to $B$ is of affine type, this mysterious region collapses to a single limiting wall.

We begin this section with some generalities on infinite rank-2 type.
Although we ultimately handle only the affine cases, we start in the general case, to give a unified framework for the affine cases and to highlight the difficulties encountered in other cases.
We use Theorem~\ref{clus mon pop root} and a computation of certain limits of ratios of cluster variables to find the function attached to the limiting wall in the affine case, without any need for representation-theoretic or algebraic-geometric machinery.
Computations related to our limit computations are found in \cite{CaSch,Zhang}.
See Remarks~\ref{UMN REU} and~\ref{Ralf and Ilke}.
The skew-symmetric case of our result recovers the affine case of a general formula conjectured in \cite[Section~1.4]{KS08} and proved in \cite[Section~6]{Reineke2}.
(See also \cite[Example~1.15]{GHKK}.)
In the non-skew-symmetric affine case, our result may be new.

\subsection{Rank-$2$ infinite type}  \label{rk2 inf sec}
We now consider the transposed cluster scattering diagram with principal coefficients for the exchange matrix $B=\begin{bsmallmatrix}0&b\\a&0\end{bsmallmatrix}$ with $ab\le-4$.
Up to the symmetry of swapping the indices $1$ and $2$, we may as well assume that $a<0$, so that $b>0$.
Continuing the notation above, we have $n=2$ and we have initial cluster variables $x_1$ and $x_2$.
For each $i\in\integers$, we define $x_{i+1}$ to be the cluster variable in $\A_\bullet(B)$ obtained by exchanging $x_{i-1}$ out of the cluster $\set{x_{i-1},x_i}$.
Thus, equivalently, $x_{i-1}$ is obtained by exchanging $x_{i+1}$ out of the cluster $\set{x_i,x_{i+1}}$.
Write $\g_i$ for the $\g$-vector of $x_i$.
In our pictures and verbal descriptions, we will let $\rho_1$ point to the right of the plane and $\rho_2$ point up, give them the same length \emph{in the page}, and talk about slopes of lines in the usual sense for this placement of $\rho_1$ and $\rho_2$.

We follow \cite[Section~9]{universal} in characterizing the $\g_i$ directly.  (Compare \cite[Example~1.15]{GHKK} and \cite[Section~3]{CGMMRSW}.)
We begin by defining a polynomial $P_m$ in $ab$ for each $m\ge-2$.
Set $P_{-2}=-1$ and $P_{-1}=0$, and for $m\ge0$, define
\begin{equation}\label{Pm recur}
P_m=\begin{cases}
-abP_{m-1}-P_{m-2}&\text{if }m\text{ is even, or}\\
P_{m-1}-P_{m-2}&\text{if }m\text{ is odd.}\\
\end{cases}
\end{equation}
Several of the $P_m$ are shown in Table~\ref{pm table}.
\begin{table}[ht]
\begin{tabular}{|r|||c|c|c|c|c|c|c|c|}
\hline
$m$&$-2$&$-1$&$0$&$1$&$2$&$3$&$4$&$5$\\\hline\hline
$P_m$&	$-1$	&$0$	&$1$	&$1$	&$-ab-1$&$-ab-2$&$a^2b^2+3ab+1$	&$a^2b^2+4ab+3$	\\\hline
\end{tabular}\\[3pt]
\caption{The polynomials $P_m$}
\label{pm table}
\end{table}
The polynomials $P_m$ are defined in \cite[Section~9]{universal} using a summation formula, and a specific relationship \cite[(9.10)]{universal} is described between the $P_m$ and the Chebyshev polynomials of the second kind.
The recursive definition given here for the $P_m$ follows by the defining recursion for the Chebyshev polynomials.
As observed in \cite[Section~9]{universal}, each $P_m$ is positive for $m\ge0$.
By \cite[Proposition~9.6]{universal}, the $\g$-vectors are given by
\begin{equation}\label{ba gs}
\g_i=\begin{cases}
-P_{-i-1}\rho_1-aP_{-i-2}\rho_2&\text{if }i\text{ is odd and }i\le-1,\\
-bP_{-i-1}\rho_1+P_{-i-2}\rho_2&\text{if }i\text{ is even and }i\le0,\\
-P_{i-3}\rho_1-aP_{i-2}\rho_2&\text{if }i\text{ is odd and }i\ge1,\\
-bP_{i-3}\rho_1+P_{i-2}\rho_2&\text{if }i\text{ is even and }i\ge2,\\
\end{cases}
\end{equation}
Write $C_i$ for the cone spanned by the $\g$-vectors of $x_i$ and $x_{i+1}$.
In particular, $C_1$ is the dominant chamber $D$.
We write $\c_i$ for the positive root orthogonal to $\g_i$. 
Recalling that the fundamental roots $\rho_i$ are dual to the simple \emph{co-roots} $\alpha_i\ck=\delta_i^{-1}\alpha_i$ and observing that the diagonalizing factors in this case must be 
$\delta_1=\frac{\gcd(-a,b)}{b}$ and $\delta_2=\frac{\gcd(-a,b)}{-a}$, we compute
\begin{equation}\label{ba cs}    
\c_i=\begin{cases}
P_{-i-2}\alpha_1-aP_{-i-1}\alpha_2&\text{if }i\text{ is even and }i\le-2,\\
bP_{-i-2}\alpha_1+P_{-i-1}\alpha_2&\text{if }i\text{ is odd and }i\le-1,\\
\alpha_1&\text{if }i=0,\\
\alpha_2&\text{if }i=1,\text{ or},\\
P_{i-2}\alpha_1-aP_{i-3}\alpha_2&\text{if }i\text{ is even and }i\ge2,\\
bP_{i-2}\alpha_1+P_{i-3}\alpha_2&\text{if }i\text{ is odd and }i\ge3,\\
\end{cases}
\end{equation}

The fan $\gFan(B)$ covers all of $V^*$ except the positive linear span $C_\infty$ of the vectors 
\begin{align}
\begin{split}
\g_\infty&=-2\sqrt{-ab}\,\rho_1-a(\sqrt{-ab}+\sqrt{-ab-4})\,\rho_2\quad\text{and}\\
\g_{-\infty}&=-b(\sqrt{-ab}+\sqrt{-ab-4})\,\rho_1+2\sqrt{-ab}\,\rho_2.
\end{split}
\end{align}
These vectors are in the limiting directions of $\g_i$ as $i\to\infty$ and $i\to-\infty$ respectively.
As we discuss in Section~\ref{affrk2 sec}, in the affine case (when $ab=-4$), the vectors $\g_\infty$ and $\g_{-\infty}$ are parallel, so that $C_\infty$ is a single limiting ray.

The situation (with $a=-3$ and $b=2$) is represented in Figure~\ref{transscatba}, together with indications of notational conventions that were just described and that will be given below.
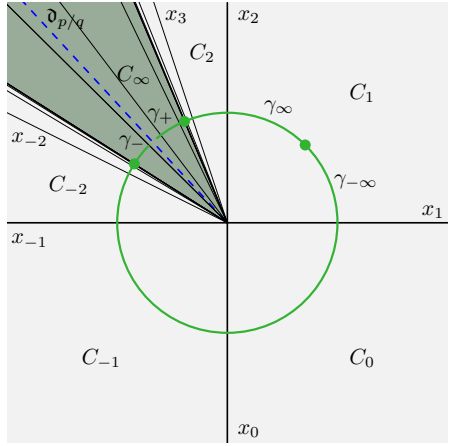
\begin{figure}
\scalebox{0.83}{\includegraphics{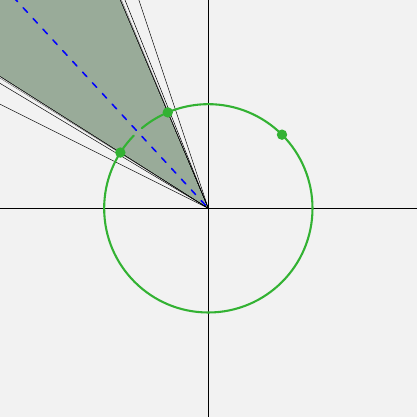}
\begin{picture}(0,0)(100,-100)
\put(-32,93){$x_3$}
\put(1,93){$x_2$}
\put(85,4){$x_1$}
\put(1,-96){$x_0$}
\put(-102,-9){$x_{-1}$}
\put(-102,37){$x_{-2}$}
\put(-21,75){$C_2$}
\put(52,57){$C_1$}
\put(52,-64){$C_0$}
\put(-70,-64){$C_{-1}$}
\put(-85,15){$C_{-2}$}
\put(-54,65){$C_\infty$}
\put(45,18){$\gamma_{-\infty}$}
\put(13,52){$\gamma_{\infty}$}
\put(-40,48){$\gamma_+$}
\put(-53,36){$\gamma_-$}
\put(-86,92){$\d_{p/q}$}
\end{picture}}
\caption{A transposed scattering diagram for an infinite rank-$2$ case}
\label{transscatba}
\end{figure}
A similar picture with $a=-2$ and $b=2$ appears later as Figure~\ref{transscat2-2}.

By Theorem~\ref{clus mon thm} and Theorem~\ref{clus easy root}, the walls of $\Scat^T(B)$ not contained in $C_\infty$ are the rays spanned by the $\g_i$, each marked with the function $(1+\hy^{\c_i})$.
The remaining possibilities for walls are the rational rays contained in $C_\infty$.
We consider a scattering diagram in the equivalence class of $\Scat^T(B)$ with one wall in each such rational ray, without assuming that the attached functions are nontrivial.
Thus, for each rational number in reduced form $p/q$ with $p>0$ and $q<0$ and with 
\begin{equation}\label{pq range}
\frac{a(\sqrt{-ab}+\sqrt{-ab-4})}{2\sqrt{-ab}}\le\frac{\,p\,}q\le\frac{2\sqrt{-ab}}{-b(\sqrt{-ab}+\sqrt{-ab-4})}\,,
\end{equation}
let $(\d_{p/q},f_{p/q})$ be the wall such that $\d_{p/q}$ is the ray spanned by $q\rho_1+p\rho_2$.
Then $f_{p/q}$ is a formal power series in $\hy^{\c_{p/q}}$, where $\c_{p/q}=\frac{bp}{\gcd(bp,aq)}\alpha_1+\frac{aq}{\gcd(bp,aq)}\alpha_2$ is the primitive element of $Q^+$ that is orthogonal to $q\rho_1+p\rho_2$.
The primitive element of $Q\ck$ parallel to $\c_{p/q}$ is $\c\ck_{p/q}=p\alpha\ck_1-q\alpha\ck_2$.

Fix $p/q$ satisfying \eqref{pq range} and choose a path $\gamma_\infty:(0,1]\to V^*\setminus\set{0}$ and a path ${\gamma_{-\infty}:(0,1]\to V^*\setminus\set{0}}$ such that
\begin{itemize} 
\item $\lim_{t\to0^+}\gamma_{\pm\infty}(t)=\g_{\pm\infty}$
\item $\gamma_{\pm\infty}(1)=\rho_1+\rho_2$
\item $\gamma_\infty$ moves in a strictly monotone-clockwise manner about $0$,
\item $\gamma_{-\infty}$ moves in a strictly monotone-counterclockwise manner about $0$.
\end{itemize}
For each $i\ge1$, let $\gamma_i$ be a subpath of $\gamma_{\infty}$ starting in the interior of $C_i$ and ending at $\rho_1+\rho_2$.
For $i\le1$, let $\gamma_i$ be a subpath of $\gamma_{-\infty}$ starting in the interior of $C_i$ and ending at $\rho_1+\rho_2$.
For consistency, take $\gamma_1$ to be the constant path at $\rho_1+\rho_2$.

We abbreviate $\p_{\gamma,\Scat^T(B)}$ as $\p_\gamma$ for any path $\gamma$.
We define $\p_\infty=\lim_{i\to\infty}\p_{\gamma_i}$.
This limit exists because for all $k$, there exists an $i_k$ such that, for $i\ge i_k$, the path obtained by deleting $\gamma_i$ from $\gamma_\infty$ crosses no wall of $\Scat^T(B)_k$.
We also define $\p_{-\infty}=\lim_{i\to-\infty}\p_{\gamma_i}$, and this limit exists for the analogous reason.
Indeed, the limit definitions are only necessary when $ab=-4$.  
When $ab<-4$, we may as well complete $\gamma_\infty$ and $\gamma_{-\infty}$ to paths from the full interval $[0,1]$ to $V^*\setminus\set{0}$. 
These paths are generic because $\g_\infty$ and $\g_{-\infty}$ are not contained in a rational ray, and the limits $\p_\infty$ and $\p_{-\infty}$ coincide with $\p_{\gamma_\infty}$ and $\p_{\gamma_{-\infty}}$.

We also define paths $\gamma_\pm:(0,1]\to V^*\setminus\set{0}$ such that
\begin{itemize} 
\item $\lim_{t\to0^+}\gamma_{\pm\infty}(t)=q\rho_1+p\rho_2$
\item $\gamma_{\pm\infty}(1)=\g_{\pm\infty}$
\item $\gamma_+$ moves in a strictly monotone-clockwise manner about $0$,
\item $\gamma_-$ moves in a strictly monotone-counterclockwise manner about $0$.
\end{itemize}
Write $\p_\pm$ for the path-ordered products obtained from the paths $\gamma_\pm$, taking appropriate limits as above.
When $ab=-4$, we think of $\gamma_\pm$ as empty paths and treat an automorphism $\p_\pm$ as the identity when it appears in formulas.

Write $\p_{p/q}$ for the wall-crossing automorphism $\p_{\gamma,\d_{p/q}}$ for a path $\gamma$ that crosses $\d_{p/q}$ with derivative $\rho_1+\rho_2$.

Since $p>0$ and $q<0$, in particular $p-q>0$.
Since $f_{p/q}$ is a (univariate) formal power series in $\hy^{\c_{p/q}}$, the following proposition determines $f_{p/q}$ completely.

\begin{prop}\label{diagonal terms ba}
For all $k\ge0$, the coefficient of $\hy^{k\c_{p/q}}$ in $f_{p/q}$ equals the coefficient of $\hy^{k\c_{p/q}}$ in $\bigl(\frac{x_1x_2}{\p_{-\infty}(\p_-(x_1x_2))}\bigr)^{\frac1{p-q}}$.
Each $\hy^{k\c_{p/q}}$ is $\hy_1^i\hy_2^j$ with $\frac ji=\frac{aq}{bp}$, and every other term in $\bigl(\frac{x_1x_2}{\p_{-\infty}(\p_-(x_1x_2))}\bigr)^{\frac1{p-q}}$ involves $\hy_1^i\hy_2^j$ with $\frac ji<\frac{aq}{bp}$.
\end{prop}

\begin{proof}
Consistency of the scattering diagram says that ${\p_{-\infty}\circ\p_-\circ\p_{p/q}^{-1}\circ\p_+^{-1}\circ\p_\infty^{-1}}$ is the identity map, so we apply it to $\p_\infty(\p_+(\p_{p/q}(x_1x_2))$ to obtain
\begin{equation}\label{diag terms ba starting point}
\p_{-\infty}(\p_-(x_1x_2))=\p_\infty(\p_+(\p_{p/q}(x_1x_2)).
\end{equation}
We calculate $\p_{p/q}x_1x_2=x_1x_2 f_\d^{\br{\rho_1+\rho_2,-p\alpha_1\ck+q\alpha_2\ck}}=x_1x_2f_{p/q}^{q-p}$.
Thus the quantity $\p_\infty(\p_+(\p_{p/q}(x_1x_2))$ is computed by starting with $x_1x_2f_{p/q}^{q-p}$ and repeatedly replacing a monomial by the same monomial times an integer power of a power series in $\hy^\beta$ for various $\beta\in Q^+$.
Specifically, each $\beta$ is of the form $c_1\alpha_1+c_2\alpha_2$ such that the slope $\frac{c_2}{c_1}$ of the ray spanned by $\beta$ is positive but strictly less than $\frac{aq}{pb}$, which is the slope of the ray spanned by $\c_{p/q}$.

Therefore, by \eqref{diag terms ba starting point}, every term in $\frac{\p_{-\infty}(\p_-(x_1x_2))}{x_1x_2}$ involves $\hy_1^i\hy_2^j$ with $\frac ji\le\frac{aq}{bp}$, and the terms where $\frac ji=\frac{aq}{bp}$ are exactly $f_{p/q}^{q-p}$.
Thus, to find $f_{p/q}$, we raise $\frac{\p_{-\infty}(\p_-(x_1x_2))}{x_1x_2}$ to the power $\frac1{q-p}$ and restrict to terms involving $\hy_1^i\hy_2^j$ with $\frac ji=\frac{aq}{bp}$.
\end{proof}

By definition, $\p_{-\infty}(\p_-(x_1x_2))$ is $\p_{-\infty}(x_1x_2\cdot\M_{p/q}(\hy_1,\hy_2))$ for some $\M_{p/q}(\hy_1,\hy_2)$ in $\k[[\hy]]$.
Thus 
\begin{equation}\label{p M}
\p_{-\infty}(\p_-(x_1x_2))=\p_{-\infty}(x_1x_2)\cdot\p_{-\infty}(\M_{p/q}(\hy_1,\hy_2)).
\end{equation}
The factor $\p_{-\infty}(\M_{p/q}(\hy_1,\hy_2))$ is difficult to deal with in general.
This is the reason we eventually restrict to the affine case, where $\p_{-\infty}(\M_{p/q}(\hy_1,\hy_2))=1$.
However, before restricting to the affine case, we make a general observation about the factor $\p_{-\infty}(x_1x_2)$.
The observation uses the following lemma, which is easily verified using \eqref{Pm recur} and~\eqref{ba gs}.
\begin{lemma}\label{deal with gs}
For $i\le-2$, \,
$\g_i=\begin{cases}
b\g_{i+1}-\g_{i+2}&\text{if }i\text{ is even, or}\\
-a\g_{i+1}-\g_{i+2}&\text{if }i\text{ is odd.}\\
\end{cases}$
\end{lemma}

We will also need a polynomial $Q_m$ in $a$ and $b$ for $m\ge0$, given by $Q_0=1$ and $Q_1=-1$ and by the recursion 
\begin{equation}\label{Qm recur}
Q_m=\begin{cases}
bQ_{m-1}-Q_{m-2}&\text{if }m\text{ is even, or}\\
-aQ_{m-1}-Q_{m-2}&\text{if }m\text{ is odd.}
\end{cases}
\end{equation}
Several values of $Q_m$ are shown in Table~\ref{qm table}.
\begin{table}[ht]
\begin{tabular}{|r||c|c|c|c|c|c|c|c|}
\hline
$m$&$0$&$1$&$2$&$3$&$4$&$5$\\\hline\hline
$Q_m$&	{\small$1$}	&{\small$-1$}	&{\small$-b-1$}&{\small$ab+a+1$}&{\small$ab^2+ab+2b+1$}&{\small$-a^2b^2-a^2b-3ab-2a-1$}	\\\hline
\end{tabular}\\[3pt]
\caption{The polynomials $Q_m$}
\label{qm table}
\end{table}

\begin{prop}\label{p inf lim}
If $F_i$ is the $F$-polynomial of the cluster variable $x_i$, then
\begin{equation}\label{p inf lim eq}
\frac{x_1x_2}{\p_{-\infty}(x_1x_2)}=\lim_{i\to-\infty}F_i^{Q_{-i-1}}\cdot F_{i+1}^{-Q_{-i}}.
\end{equation}
\end{prop}
\begin{proof}
We use Theorem~\ref{clus mon pop root} and Theorem~\ref{clus mon thm} to write $x_i=\p_{\gamma_i}(x^{\g_i})$ and $x_{i+1}=\p_{\gamma_i}(x^{\g_{i+1}})$ for all $i\in\integers$.
We are justified in using the same path $\gamma_i$ in both of these equations because $C_i$ is the cone spanned by $\g_i$ and $\g_{i+1}$.
By an easy induction using Lemma~\ref{deal with gs} and \eqref{Qm recur}, we show that $-\g_iQ_{-i-1}+\g_{i+1}Q_{-i}=\rho_1+\rho_2$ for all $i\le-1$.
Thus $x_i^{-Q_{-i-1}}x_{i+1}^{Q_{-i}}=\p_{\gamma_i}(x_1x_2)$.
Thus $\p_{-\infty}(x_1x_2)$ is $\lim_{i\to-\infty}x_i^{-Q_{-i-1}}x_{i+1}^{Q_{-i}}$, which equals $x_1x_2\lim_{i\to-\infty}F_i^{-Q_{-i-1}}\cdot F_{i+1}^{Q_{-i}}$.
\end{proof}

\subsection{Rank-2 affine type}\label{affrk2 sec}
In the affine cases (when $ab=-4$), the considerations of Section~\ref{rk2 inf sec} are sufficient to determine the function attached to the limiting wall.
We will prove the following theorem.
The remaining affine cases can be obtained from the theorem by Proposition~\ref{scat antip} and/or by swapping the indices $1$ and $2$.

\begin{thm}\label{rk2 aff formula}
The function on the limiting wall of $\Scat^T\bigl(\begin{bsmallmatrix}\,\,\,\,\,0&\,2\\-2&\,0\end{bsmallmatrix}\bigr)$ is $\displaystyle\frac1{(1-\hy_1\hy_2)^2}$.
The function on the limiting wall of $\Scat^T\bigl(\begin{bsmallmatrix}\,\,\,\,\,0&\,1\\-4&\,0\end{bsmallmatrix}\bigr)$ is $\displaystyle\frac{1+\hy_1\hy_2^2}{(1-\hy_1\hy_2^2)^2}$.
\end{thm}

We start with the first assertion, which is the case $a=-2$ and $b=2$ in the notation of Section~\ref{rk2 inf sec}.
Figure~\ref{transscat2-2} shows $\Scat^T(B)$ in this case.
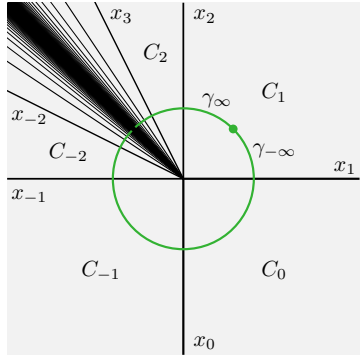
\begin{figure}
\scalebox{0.83}{\includegraphics{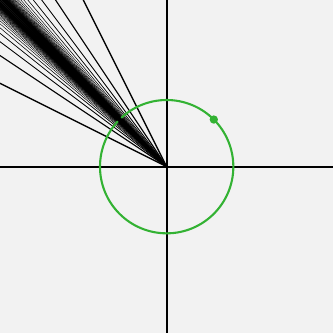}
\begin{picture}(0,0)(80,-80)
\put(-37,73){$x_3$}
\put(1,73){$x_2$}
\put(65,4){$x_1$}
\put(1,-76){$x_0$}
\put(-82,-9){$x_{-1}$}
\put(-82,27){$x_{-2}$}
\put(-22,55){$C_2$}
\put(32,37){$C_1$}
\put(32,-44){$C_0$}
\put(-50,-44){$C_{-1}$}
\put(-65,10){$C_{-2}$}
\put(28,12){$\gamma_{-\infty}$}
\put(5,35){$\gamma_{\infty}$}
\end{picture}}
\caption{A transposed scattering diagram for the skew-symmetric affine rank-$2$ case}
\label{transscat2-2}
\end{figure}
For these values of $a$ and $b$, the slope of the limiting ray is $p/q=-1$, the positive root orthogonal to that ray is $\c_{-1}=\alpha_1+\alpha_2$, and $Q_m=-2m+1$.
Combining Propositions~\ref{diagonal terms ba} and~\ref{p inf lim} in this case, we obtain the following proposition.
\begin{prop}\label{diagonal terms}
If $a=-2$ and $b=2$, then for all $k\ge0$, the coefficient of $\hy_1^k\hy_2^k$ in $f_{-1}$ equals the coefficient of $\hy_1^k\hy_2^k$ in $\sqrt{\lim_{i\to-\infty}F_i^{2i+3}\cdot F_{i+1}^{-2i-1}}$.
\end{prop}

We will call a term \newword{diagonal} if it is a constant times $\hy_1^k\hy_2^k$ and \newword{super-diagonal} if it is a constant times $\hy_1^j\hy_2^k$ for $k>j$.
The following lemma will let us evaluate the limit in Proposition~\ref{diagonal terms}.

\begin{lemma}\label{F coeffs}
When $a=-2$ and $b=2$, the $F$-polynomial $F_i$ has the following properties for $i\le 0$.
\begin{enumerate}[\rm(i)]
\item \label{leading}
As a polynomial in $\hy_1$ (with coefficients polynomials in $\hy_2$), the leading term is $\hy_1^{-i}(1+\hy_2)^{-i+1}$.
\item \label{others}
The only super-diagonal term is $\hy_1^{-i}\hy_2^{-i+1}$.
\item \label{key terms}
The diagonal terms are $(j+1)\hy_1^j\hy_2^j$ for $0\le j\le-i$.
\end{enumerate}
\end{lemma}
\begin{proof}
We argue by induction on $-i$.
Unwrapping \cite[Proposition~5.1]{ca4} in this case, we obtain $F_0=1+\hy_2$ and $F_{-1}=1+\hy_1(1+\hy_2)^2$, and for $i\le-2$ we obtain
\begin{equation}\label{F recur}
F_i=\frac{F_{i+1}^2+\hy_1^{-i-2}\hy_2^{-i-1}}{F_{i+2}}.
\end{equation}
(The exponents on $\hy_1$ and $\hy_2$ come from \eqref{ba cs}.)
We have established the base cases $i=0$ and $i=-1$, and for the rest of the proof, we take $i<-1$.
By induction and by dividing leading terms, assertion \eqref{leading} follows easily from \eqref{F recur}.

To prove the remaining assertions, we track the diagonal and super-diagonal terms through \eqref{F recur}.
We perform long division and, since we know that the result is a polynomial, we can ignore terms that are not diagonal or super-diagonal.
Indeed, because by induction we know that the highest order term in the denominator of \eqref{F recur} is super-diagonal, we need only track super-diagonal terms in the numerator.
By induction, the relevant terms in the numerator are 
\[\hy_1^{-2i-2}\hy_2^{-2i}+\hy_1^{-i-2}\hy_2^{-i-1}+\hy_1^{-i-1}\hy_2^{-i}\sum_{j=0}^{-i-1}2(j+1)\hy_1^j\hy_2^j\]
and the relevant terms in the denominator are $\hy_1^{-i-2}\hy_2^{-i-1}+\sum_{j=0}^{-i-2}(j+1)\hy_1^j\hy_2^j$.
Division yields super-diagonal and diagonal terms $\hy_1^{-i}\hy_2^{-i+1}+\sum_{j=0}^{-i}(j+1)\hy_1^j\hy_2^j$ as desired.
\end{proof}

Lemma~\ref{F coeffs}\eqref{key terms} implies that, for $i<0$, the restriction of $F_i$ to diagonal terms agrees with the formal power series $\sum_{k\ge0}(k+1)\hy_1^k\hy_2^k=(1-\hy_1\hy_2)^{-2}$ up to the term $\hy_1^{-i}\hy_2^{-i}$.
By Lemma~\ref{F coeffs}\eqref{others}, the only super-diagonal terms in $F_i$ have degree at least $-i$ in $\hy_1$, so the diagonal terms of $F_{i}^{2i+3}\cdot F_{i+1}^{-2i-1}$ agree, up to $\hy_1^{-i-1}\hy_2^{-i-1}$, with $(1-\hy_1\hy_2)^{-4}$.
Thus Proposition~\ref{diagonal terms} establishes the first assertion of Theorem~\ref{rk2 aff formula}.

\begin{remark}\label{UMN REU}
In \cite{Zhang}, there is a formal power series closely related to the function $f_{-1}$ on the limiting wall.
It appears as a limit of ``stable cluster variables'' (certain transformed $F$-polynomials). 
\end{remark}

We next prove the second assertion of Theorem~\ref{rk2 aff formula}.
In this case,  $a=-4$ and $b=1$, the slope of the limiting ray is $p/q=-2$, the positive root orthogonal to that ray is $\c_{-2}=\alpha_1+2\alpha_2$, and $Q_m$ is $-\frac32m+1$ for $m$ even or $-3m+2$ for $m$ odd.
We know that the limit in Proposition~\ref{p inf lim} exists, so we are free to approach it on the even values of $i$.
We will do this by replacing $i$ by $2i$ in the limit formula.
Thus Propositions~\ref{diagonal terms ba} and~\ref{p inf lim} give the following proposition.
\begin{prop}\label{diagonal terms asym}
If $a=-4$ and $b=1$, then for all $k\ge0$, the coefficient of $\hy_1^k\hy_2^{2k}$ in $f_{-2}$ equals the coefficient of $\hy_1^k\hy_2^{2k}$ in $\sqrt[3]{\lim_{i\to-\infty}F_{2i}^{6i+5}\cdot F_{2i+1}^{-3i-1}}$.
\end{prop}

Again, evaluating the limit will require determining some relevant coefficients of the $F$-polynomials, but this time we must separate the even and odd indices.
We will also \emph{re-use} the words diagonal and super-diagonal to deal with this case.
We call a term \newword{diagonal} if it is a constant times $\hy_1^k\hy_2^{2k}$ and \newword{super-diagonal} if it is a constant times $\hy_1^k\hy_2^\ell$ for $\ell>2k$.
Furthermore, a term is \newword{super-super-diagonal} if it is a constant times $\hy_1^k\hy_2^\ell$ for $\ell>2k+1$.

\begin{lemma}\label{F coeffs asym}
Suppose $a=-4$ and $b=1$.
The even-indexed $F$-polynomials $F_{2i}$ for $i\le0$ have the following properties:
\begin{enumerate}[\rm(i)]
\item \label{leading asym}
As a polynomial in $\hy_1$ (with coefficients polynomials in $\hy_2$), the leading term is $\hy_1^{-i}(1+\hy_2)^{-2i+1}$.
\item \label{super asym}
The only super-diagonal term is $\hy_1^{-i}\hy_2^{-2i+1}$.
\item \label{key terms asym}
The diagonal terms are $\sum_{j=0}^{-i}(2j+1)\hy_1^j\hy_2^{2j}$.
\end{enumerate}
The odd-indexed $F$-polynomials $F_{2i+1}$ for $i\le-1$ have the following properties:
\begin{enumerate}[\rm(i)] \setcounter{enumi}{3}
\item \label{leading asym odd}
As a polynomial in $\hy_1$ (with coefficients polynomials in $\hy_2$), the leading term is $\hy_1^{-2i-1}(1+\hy_2)^{-4i}$.
\item \label{super asym odd}
The only super-diagonal terms are $\hy_1^{-2i-1}\hy_2^{-4i}+\sum_{j=1}^{-i}4j\hy_1^{-i+j-1}\hy_2^{-2i+2j-1}$.
\item \label{key terms asym odd}
The diagonal terms agree with $\bigl(\sum_{j\ge0}(2j+1)\hy_1^j\hy_2^{2j}\bigr)^2$ up to $(-2i+1)\hy_1^{-i}\hy_2^{-2i}$.
\end{enumerate}
\end{lemma}

\begin{proof}
We argue by induction on $m\le0$ that $F_m$ has the properties described in the proposition.
In this case \cite[Proposition~5.1]{ca4} says that $F_0=1+\hy_2$ and ${F_{-1}=1+\hy_1(1+\hy_2)^4}$.
Also, for $i\le -1$
\begin{equation}\label{F recur asym even}
F_{2i}=\frac{F_{2i+1}+\hy_1^{-i-1}\hy_2^{-2i-1}}{F_{2i+2}}.
\end{equation}
and for $i\le-2$,
\begin{equation}\label{F recur asym odd}
F_{2i+1}=\frac{(F_{2i+2})^4+\hy_1^{-2i-3}\hy_2^{-4i-4}}{F_{2i+3}}.
\end{equation}
(The monomials $\hy_1^{-i-1}\hy_2^{-2i-1}$ and $\hy_1^{-2i-3}\hy_2^{-4i-4}$ are $\hy^{\c_{m-1}}$ for $m=2i$ or $2i+1$ and $\c_{m-1}$ as in \eqref{ba cs}.)
We have established the base cases $m=0$ and $m=-1$, and for the rest of the proof, we take $m<-1$.
Assertions \eqref{leading asym} and \eqref{leading asym odd} follow easily by induction.

To prove the remaining assertions, we track diagonal and super-diagonal terms through the right sides of \eqref{F recur asym even} and \eqref{F recur asym odd}.
Since we know that the right side is a polynomial, we can simply perform polynomial long division and ignore all terms that are not diagonal or super-diagonal.
Because we know by induction the highest-order terms of the denominator of \eqref{F recur asym even}, we need only concern ourselves with super-diagonal terms in the numerator.
Similarly, in the numerator of \eqref{F recur asym odd}, we need only concern ourselves with super-super-diagonal terms.

For $m=2i$, by induction the relevant terms in the numerator in \eqref{F recur asym even} are 
\begin{equation}\label{even numer}
\hy_1^{-2i-1}\hy_2^{-4i}+\sum_{j=1}^{-i}4j\hy_1^{-i+j-1}\hy_2^{-2i+2j-1}+\hy_1^{-i-1}\hy_2^{-2i-1}
\end{equation}
and the relevant terms in the denominator are
\begin{equation}\label{even denom}
\hy_1^{-i-1}\hy_2^{-2i-1}+\sum_{j=0}^{-i-1}(2j+1)\hy_1^j\hy_2^{2j}.
\end{equation}
The relevant terms of the quotient are $\hy_1^{-i}\hy_2^{-2i+1}+\sum_{j=0}^{-i}(2j+1)\hy_1^j\hy_2^{2j}$, as desired.

For $m=2i+1$, by induction, the relevant terms in the numerator are the super-super-diagonal terms in $\bigl(\hy_1^{-i-1}\hy_2^{-2i-1}+\sum_{j=0}^{-i-1}(2j+1)\hy_1^j\hy_2^{2j}\bigr)^4+\hy_1^{-2i-3}\hy_2^{-4i-4}$, namely
\begin{multline}\label{odd numer}
\hy_1^{-4i-4}\hy_2^{-8i-4}+4\hy_1^{-3i-3}\hy_2^{-6i-3}\sum_{j=0}^{-i-1}(2j+1)\hy_1^j\hy_2^{2j}\\
+6\hy_1^{-2i-2}\hy_2^{-4i-2}\biggl(\sum_{j=0}^{-i-1}(2j+1)\hy_1^j\hy_2^{2j}\biggr)^2
+\hy_1^{-2i-3}\hy_2^{-4i-4}.
\end{multline}
The relevant terms in the denominator are
\begin{multline}\label{odd denom}
\hy_1^{-2i-3}\hy_2^{-4i-4}+\hy_1^{-i-2}\hy_2^{-2i-3}\sum_{j=0}^{-i-1}4j\hy_1^j\hy_2^{2j}\\
+\biggl(\sum_{j=0}^{-i-1}(2j+1)\hy_1^j\hy_2^{2j}\biggr)^2+\hy_1^{-i}\hy_2^{-2i}\sum_{j=0}^{-i-3}b_j\hy_1^j\hy_2^{2j},
\end{multline}
for some coefficients $b_j$ that (it will turn out) are not important.
The super-diagonal terms in the quotient are $\hy_1^{-2i-1}\hy_2^{-4i}+4\hy_1^{-i-1}\hy_2^{-2i-1}\sum_{j=0}^{-i}j\hy_1^j\hy_2^{2j}$ and the diagonal terms are
\begin{multline}\label{X}
1+(6\hy_1\hy_2^2-\hy_1^2\hy_2^4)\biggl(\sum_{j=0}^{-i-1}(2j+1)\hy_1^j\hy_2^{2j}\biggr)^2\\-16\biggl(\sum_{j=0}^{-i-1}j\hy_1^j\hy_2^{2j}\biggr)\biggl(\sum_{j=0}^{-i}j\hy_1^j\hy_2^{2j}\biggr)+\hy_1^{-i+2}\hy_2^{-i+4}\biggl(\sum_{j=0}^{-i-1}b_j\hy_1^j\hy_2^{2j}\biggr).
\end{multline}
But \eqref{X} agrees, up to terms involving $\hy_1^{-i}\hy_2^{-2i}$ with 
\begin{equation}\label{X FPS}
1+(6\hy_1\hy_2^2-\hy_1^2\hy_2^4)\biggl(\frac{1+\hy_1\hy_2^2}{(1-\hy_1\hy_2^2)^2}\biggr)^2\\-16\biggl(\frac{\hy_1\hy_2^2}{(1-\hy_1\hy_2^2)^2}\biggr)=\biggl(\frac{1+\hy_1\hy_2^2}{(1-\hy_1\hy_2^2)^2}\biggr)^2,
\end{equation}
which equals $=\bigl(\sum_{j\ge0}(2j+1)\hy_1^j\hy_2^{2j}\bigr)^2$ as desired.
\end{proof}

Lemma~\ref{F coeffs asym} says in particular that for $i\le-1$ the diagonal terms of $F_{2i}$ agree, up to $\hy_1^{-i}\hy_2^{-2i}$, with $\frac{1+\hy_1\hy_2^2}{(1-\hy_1\hy_2^2)^2}$ and the diagonal terms of $F_{2i+1}$ agree, up to $\hy_1^{-i}\hy_2^{-2i}$, with $\Bigl(\frac{1+\hy_1\hy_2^2}{(1-\hy_1\hy_2^2)^2}\Bigr)^2$.
Furthermore, the only super-diagonal terms in either polynomial have degree at least $-i$ in $\hy_1$, so the diagonal terms of $F_{2i}^{6i+5}\cdot F_{2i+1}^{-3i-1}$ agree, up to $\hy_1^{-i-1}\hy_2^{-2i-2}$, with $\Bigl(\frac{1+\hy_1\hy_2^2}{(1-\hy_1\hy_2^2)^2}\Bigr)^{6i+5}\cdot\Bigl(\frac{1+\hy_1\hy_2^2}{(1-\hy_1\hy_2^2)^2}\Bigr)^{-6i-2}$, which equals $\Bigl(\frac{1+\hy_1\hy_2^2}{(1-\hy_1\hy_2^2)^2}\Bigr)^3$.
Now Proposition~\ref{diagonal terms asym} completes the proof of Theorem~\ref{rk2 aff formula}.


\subsection{Path-ordered products and Narayana numbers}\label{Nar sec} 
As an easy consequence of Corollary~\ref{clus mon F}, for each ray of $\gFan(B)$, the $F$-polynomial of the corresponding cluster variable is $x^{-\lambda}\cdot\p_{\gamma,\Scat^T(B)}(x^\lambda)$, where $\lambda$ is the corresponding $\g$-vector, ${\gamma:(0,1]\to V^*}$ is any generic path contained in the union of the cones of $\gFan(B)$ such that $\lim_{t\to0^+}\gamma(t)$ is in the relative interior of the ray and $\gamma(1)$ is in the positive cone.
(We interpret the path-ordered product $\p_{\gamma,\Scat^T(B)}$ as a limit as in Section~\ref{rk2 inf sec}.)


For $B$ of rank $2$ and of affine type, we can write down the same formula for $\lambda$ in the limiting ray.
We obtain not a polynomial, but a formal power series in $\k[[\hy]]$.
In this section, we compute this series in the symmetric rank-$2$ affine case.

We work in the notation of Section~\ref{rk2 inf sec}, with $a=-2$ and $b=2$.
We take $\lambda=-\rho_1+\rho_2$, so that $x^\lambda=x_1^{-1}x_2$.
Since the path-ordered products $\p_+$ and $\p_-$ are the identity in this case, consistency says that $\p_{-\infty}\circ\p_{p/q}^{-1}\circ\p_\infty^{-1}$ is the identity map.
Since also $\p_{p/q}x_1^{-1}x_2=x_1^{-1}x_2$, we have $\p_{-\infty}(x_1^{-1}x_2)=\p_\infty(x_1^{-1}x_2)$.
We define $\N(\hy_1,\hy_2)\in\k[[\hy]]$ to be the formal power series such that $\p_{-\infty}(x_1^{-1}x_2)=\p_\infty(x_1^{-1}x_2)=x_1^{-1}x_2\N(\hy_1,\hy_2)$.
The symbol ``$\N$\hspace{0.7pt}'' is to suggest Narayana:  We now prove that the coefficients of $\N(\hy_1,\hy_2)$ are given by Narayana numbers with an alternating sign.
For $i,j\ge 0$, the \newword{Narayana number} is 
\begin{equation}\label{Nar def}
\Nar(i,j)=\begin{cases}
1&\text{if }i=j=0,\\
0&\text{otherwise if }ij=0,\text{ or}\\
\frac1i\binom ij\binom i{j-1}&\text{if }i\ge1\text{ and }j\ge1.
\end{cases}
\end{equation}

We will prove the following theorem.

\begin{theorem}\label{Nar thm}
\begin{equation}\label{Nar thm eq}
\N(\hy_1,\hy_2)=\lim_{i\to\infty}\frac{F_{i+1}}{F_i}=\lim_{i\to-\infty}\frac{F_{i-1}}{F_i}=1+\hy_1\sum_{i,j\ge0}(-1)^{i+j}\Nar(i,j)\hy_1^i\hy_2^j.
\end{equation}
\end{theorem}

Using the known generating function for the Narayana numbers---see for example \cite[Chapter 2]{PetersenBook}, where the indexing conventions are slightly different---we have the following immediate corollary of Theorem~\ref{Nar thm}.

\begin{corollary}\label{Nar cor}
\begin{equation}\label{Nar cor eq}
\N(\hy_1,\hy_2)=\frac{1+\hy_1+\hy_1\hy_2+\sqrt{(1+\hy_1+\hy_1\hy_2)^2-4\hy_1\hy_2}}{2}.
\end{equation}
\end{corollary}

We now proceed to prove Theorem~\ref{Nar thm}.
As before (by Theorem~\ref{clus mon pop root} and Theorem~\ref{clus mon thm}), $x_i=\p_{\gamma_i}(x^{\g_i})$ and $x_{i+1}=\p_{\gamma_i}(x^{\g_{i+1}})$ for all $i\in\integers$.
Since $a=-2$ and $b=2$, by \eqref{ba gs} we have $\g_{i+1}-\g_i=-\rho_1+\rho_2$ for $i\ge1$ and $\g_i-\g_{i+1}=-\rho_1+\rho_2$ for $i\le-2$.
We see that 
\begin{equation}\label{pi F's}
\p_{\gamma_i}(x_1^{-1}x_2)=\begin{cases}
x_i^{-1}\cdot x_{i+1}=\,x_1^{-1}x_2F_i^{-1}\cdot F_{i+1}&\text{if }i\ge1,\text{ or}\\
x_{1+1}^{-1}\cdot x_i\hspace{6.25pt}=\,x_1^{-1}x_2F_{i+1}^{-1}\cdot F_i&\text{if }i\le-2,
\end{cases}
\end{equation}
where $F_i$ is the $F$-polynomial of the cluster variable $x_i$.
The first two equalities of \eqref{Nar thm eq} follow.

The remaining assertion of Theorem~\ref{Nar thm} is illustrated by the following two-dimensional representation of $\N(\hy_1,\hy_2)$.
{\small\[\begin{array}{clrrrrrr}
1\\
+\hy_1\\
&+\hy_1^2\hy_2\\
&-\hy_1^3\hy_2&+\hy_1^3\hy_2^2\\\
&+\hy_1^4\hy_2&-3\hy_1^4\hy_2^2&+\hy_1^4\hy_2^3\\\
&-\hy_1^5\hy_2&+6\hy_1^5\hy_2^2&-6\hy_1^5\hy_2^3&+\hy_1^5\hy_2^4\\\
&+\hy_1^6\hy_2&-10\hy_1^6\hy_2^2&+20\hy_1^6\hy_2^3&-10\hy_1^6\hy_2^4&+\hy_1^6\hy_2^5\\\
&-\hy_1^7\hy_2&+15\hy_1^7\hy_2^2&-50\hy_1^7\hy_2^3&+50\hy_1^7\hy_2^4&-15\hy_1^7\hy_2^5&+\hy_1^7\hy_2^6\\\
&+\hy_1^8\hy_2&-21\hy_1^8\hy_2^2&+105\hy_1^8\hy_2^3&-175\hy_1^8\hy_2^4&+105\hy_1^8\hy_2^5&-21\hy_1^8\hy_2^6&+\hy_1^8\hy_2^7
\\
&+\cdots
\end{array}
\]}

To prove this remaining assertion, we begin with the following two propositions, which amount to boundary conditions and a functional equation on $\N(\hy_1,\hy_2)$.
The first of the propositions is analogous to an observation already made in the proof of Proposition~\ref{diagonal terms ba}.
\begin{prop}\label{simple Nar fact}
$\N(\hy_1,\hy_2)$ is $1$ plus terms involving $\hy_1^i\hy_2^j$ with $i>j$.
\end{prop}
\begin{proof} 
$\p_\infty(x_1^{-1}x_2)$ is computed by starting with $x_1^{-1}x_2$ and repeatedly replacing a monomial by the same monomial times integer powers of power series in $\hy^\beta$ for $\beta\in Q^+$ of the form $\beta=c_1\alpha_1+c_2\alpha_2$ with $0<\frac{c_2}{c_1}<1$.
(See Figure~\ref{transscat2-2}.)
\end{proof}

\begin{prop}\label{Nar fun}
\begin{equation}\label{Nar fun eq}
\N\bigl(\hy_1(1+\hy_2)^{-2},\hy_2\bigr)\cdot(1+\hy_2)=\N\bigl(\hy_2(1+\hy_1)^{-2},\hy_1\bigr)\cdot(1+\hy_1).
\end{equation}
\end{prop}
\begin{proof}
Define $\tilde{x}_i$ to be $x_{1-i}$ for all $i\in\integers$.
Then $\p_{\gamma_i}(x_1^{-1}x_2)=\tilde{x}_{1-i}^{-1}\cdot\tilde{x}_{2-i}$ for $i\ge 1$, so $\N(\hy_1,\hy_2)x_1^{-1}x_2=\lim_{i\to-\infty}\tilde{x}_i\cdot\tilde{x}_{i+1}^{-1}$.
The cluster $\set{\tilde{x}_1,\tilde{x}_2}=\set{x_0,x_{-1}}$ can be taken to be the initial cluster in a cluster algebra with exchange matrix $B=\begin{bsmallmatrix}\,\,\,\,\,0&\,2\\-2&\,0\end{bsmallmatrix}$, but with non-principal coefficients.
However, we can ignore coefficients by setting $y_1=y_2=1$, or equivalently setting $\hy_1=x_2^{-2}$ and $\hy_2=x_1^2$.
Making this substitution on both sides of the equation $\N(\hy_1,\hy_2)\cdot x_1^{-1}x_2=\lim_{i\to-\infty}\tilde{x}_i\cdot\tilde{x}_{i+1}^{-1}$ and interpreting the new right side, we obtain
\begin{equation}\label{Nar functional eq 1}
\N(x_2^{-2},x_1^2)\cdot x_1^{-1}x_2=\N(\tilde{x}_2^{-2},\tilde{x}_1^2)\cdot\tilde{x}_1^{-1}\tilde{x}_2.
\end{equation}
Using the usual (coefficient-free) exchange relations \cite[(2.15)]{ca4}, we compute
\[\tilde{x}_1=x_0=\frac{1+x_1^2}{x_2}\,,\quad\text{and}\quad\tilde{x}_2=x_{-1}=\frac{1+(1+x_1^2)^2x_2^{-2}}{x_1}.\]
Thus \eqref{Nar functional eq 1} becomes
\begin{multline}\label{Nar functional eq 2}
\N(x_2^{-2},x_1^2)\cdot x_1^{-1}x_2\\=\N\biggl(\frac{x_1^2}{(1+(1+x_1^2)^2x_2^{-2})^2},\frac{(1+x_1^2)^2}{x_2^2}\biggr)\cdot\frac{x_2}{1+x_1^2}\cdot\frac{1+(1+x_1^2)^2x_2^{-2}}{x_1}.
\end{multline}
Canceling $x_1^{-1}x_2$ and replacing $x_2^{-2}$ by $\hy_1$ and $x_1^2$ by $\hy_2$ in \eqref{Nar functional eq 2}, we obtain
\begin{equation}\label{Nar functional eq 3}
\N(\hy_1,\hy_2)=\N\biggl(\frac{\hy_2}{(1+(1+\hy_2)^2\hy_1)^2},\hy_1(1+\hy_2)^2\biggr)\frac{1+(1+\hy_2)^2\hy_1}{1+\hy_2}.
\end{equation}
Finally, replacing $\hy_1$ by $\hy_1(1+\hy_2)^{-2}$ in \eqref{Nar functional eq 3}, we obtain
the proposition.
\end{proof}

\begin{proof}[Proof of Theorem~\ref{Nar thm}]
Write $\N(\hy_1,\hy_2)=\sum_{i,j\ge0}n_{ij}\hy_1^i\hy_2^j$.
Extracting the coefficient of $\hy_1^i\hy_2^j$ on both sides of \eqref{Nar fun eq}, we obtain
\begin{equation}\label{Nar recur prelim}
\sum_{k=0}^jn_{ik}\begin{psmallmatrix}-2i\\\\j-k\end{psmallmatrix}+\sum_{k=0}^{j-1}n_{ik}\begin{psmallmatrix}{-2i}\\\\{j-1-k}\end{psmallmatrix}
=\sum_{k=0}^in_{jk}\begin{psmallmatrix}-2j\\\\i-k\end{psmallmatrix}+\sum_{k=0}^{i-1}n_{jk}\begin{psmallmatrix}-2j\\\\i-1-k\end{psmallmatrix}
\end{equation}
Since $\binom{z}{-1}=0$ for all $z\neq-1$, we can extend the second sum on each side (allowing $k=j$ or $k=i$ respectively) and combine the two sums using the identity $\binom{z}{w}+\binom{z}{w-1}=\binom{z+1}{w}$ to obtain
\begin{equation}\label{Nar recur}
\sum_{k=0}^jn_{ik}\begin{psmallmatrix}-2i+1\\\\j-k\end{psmallmatrix}
=\sum_{k=0}^in_{jk}\begin{psmallmatrix}-2j+1\\\\i-k\end{psmallmatrix}
\end{equation}
To determine the $n_{ij}$ from \eqref{Nar recur}, we need ``boundary conditions.''
Proposition~\ref{simple Nar fact} says that $n_{00}=1$ and that otherwise $n_{ij}=0$ for $i\le j$.
Thus for $i>j$, we see that \eqref{Nar recur} writes $n_{ij}$ in terms of other nonzero coefficients $n_{i'j'}$ with $j'<j$ (or in terms of $n_{00}=1$), and thus uniquely determines $n_{ij}$ by induction.
Therefore, to complete the proof, it is enough to verify that the coefficients described in Theorem~\ref{Nar thm} satisfy \eqref{Nar recur}.

When $i>j=0$, the recursion \eqref{Nar recur} says that $n_{i0}=n_{00}\binom1i$, which is $1$ if $i=1$ and zero if $i>1$.
Thus when $i>j=1$, the recursion says that $n_{i1}=n_{10}\binom{-1}{i}=(-1)^i$.  
When $i>j>1$, the $k=0$ terms of the recursion are zero and the nonzero terms of the recursion are $n_{\ell m}$ with $\ell\ge m>1$.
We want to show that when we replace each such $n_{\ell m}$ by $(-1)^{\ell+m+1}\frac1{\ell-1}\binom{\ell-1}m\binom{\ell-1}{m-1}$, we get equality.
We rewrite each $\frac1{\ell-1}\binom{\ell-1}m\binom{\ell-1}{m-1}$ as $\frac{(2-\ell)_{m-1}(1-\ell)_{m-1}}{(2)_{m-1}(m-1)!}$.
Here $(x)_{m-1}$ denotes---as is standard in the context of hypergeometric series---the \emph{rising} factorial $x(x+1)\cdots(x+m-2)$.
We also rewrite $\binom{-2i+1}{j-k}$ as $(-1)^{k-1}\binom{-2i+1}{j-1}\frac{(1-j)_{k-1}}{(-2i-j+3)_{k-1}}$ and $\binom{-2j+1}{i-k}$ as $(-1)^{k-1}\binom{-2j+1}{i-1}\frac{(1-i)_{k-1}}{(-2j-i+3)_{k-1}}$.
Thus the identity we must check is 
\begin{multline}\label{Nar must check}
(-1)^i\binom{-2i+1}{j-1}\sum_{k=1}^j\frac{(2-i)_{k-1}(1-i)_{k-1}(1-j)_{k-1}}{(2)_{k-1}(-2i-j+3)_{k-1}(k-1)!}\\
=(-1)^j\binom{-2j+1}{i-1}\sum_{k=1}^i\frac{(2-j)_{k-1}(1-j)_{k-1}(1-i)_{k-1}}{(2)_{k-1}(-2j-i+3)_{k-1}(k-1)!}.
\end{multline}
In standard hypergeometric series notation, this is 
\begin{multline}\label{Nar hyper}
(-1)^i\binom{-2i+1}{j-1}\,_3F_2\left[\begin{smallmatrix}2-i&1-i&1-j\\2&-2i-j+3\\\end{smallmatrix};1\right]\\
=(-1)^j\binom{-2j+1}{i-1}\,_3F_2\left[\begin{smallmatrix}2-j&1-j&1-i\\2&-2j-i+3\\\end{smallmatrix};1\right].
\end{multline}
This is verified by applying Saalsch\"{u}tz' Theorem to both sides and making some simple manipulations to simplify both sides to $(-1)^{i+j-1}\frac{1}{i+j}\binom{i+j-2}{j-1}\binom{i+j}{j}$.
\end{proof}


\begin{remark}\label{Ralf and Ilke}
We thank Gregg Musiker for pointing out work of Canakci and Schiffler \cite{CaSch}, which provides a different way of computing the limit in Theorem~\ref{Nar thm}.
Rewritten in our notation (including switching the roles of $x_1$ and $x_2$), \cite[Corollary~7.6(b)]{CaSch} says 
\begin{equation}
\left[\lim_{i\to-\infty}\frac{x_{i-1}}{x_i}\right]_{y_1=y_2=1}=\left[\frac{x_2+x_0+\sqrt{(x_2-x_0)^2+4}}{2x_1}\right]_{y_1=y_2=1}.\end{equation}
We can put back the coefficients $y_1$ and $y_2$:
Each $\frac{x_{i+1}}{x_i}$ has $\g$-vector $-\rho_1+\rho_2$, so we insert coefficients so as to make the right side homogeneous with that $\g$-vector.
Using the fact that the $\g$-vector of $y_1$ is $2\rho_2$ 
and the $\g$-vector of $x_0$ is $-\rho_2$, we see that 
\begin{equation}\lim_{i\to-\infty}\frac{x_{i-1}}{x_i}=\frac{x_2+y_1x_0+\sqrt{(x_2-y_1x_0)^2+4y_1}}{2x_1}.\end{equation}
Now, using the fact that $y_1=\hy_1x_2^2$ and $x_0=\frac{1+\hy_2}{x_2}$ and then multiplying both sides by $x_1x_2^{-1}$, 
we see that $\lim_{i\to-\infty}\frac{F_{i-1}}{F_i}$ equals the right side of \eqref{Nar cor eq}.
\end{remark}

\begin{remark}\label{anon}
We thank an anonymous referee for pointing out the relationship between $\N(\hy_1,\hy_2)$ and the theta function $\thet_{-\rho_1+\rho_2}$, leading to yet another way to compute $\N(\hy_1,\hy_2)$.
We sketch the argument here.
In Section~\ref{trans clus sec}, we quoted the definition of theta functions $\thet_\lambda$ in terms of a generic point $p$ in the dominant chamber $D$.
As mentioned, $\thet_\lambda$ does not depend on the choice of $p$ as long as it is generic and in $D$.
More generally, the definition given there applies to define $\thet_{\lambda,p}$ for any point $p$ not in the support of the scattering diagram.
Taking $\lambda=-\rho_1+\rho_2$ and taking $p$ ``very close'' to the limiting ray, one can convince oneself that there are only two broken lines and that $\thet_{\lambda,p}=x_1^{-1}x_2(1+\hy_1\hy_2)$.
The result of \cite[Section~4]{CPS} mentioned in Section~\ref{trans clus sec} says that for fixed $\lambda$, the theta functions for different choices of $p$ are related by path-ordered products.
(See also \cite[Theorem~3.5]{GHKK}.)
One can use that result to compute $\thet_\lambda=\p_\infty(x_1^{-1}x_2(1+\hy_1\hy_2))$.
The latter is equal to $x_1^{-1}x_2\bigl(\N(\hy_1,\hy_2)+\frac{\hy_1\hy_2}{\N(\hy_1,\hy_2)}\bigr)$.
By putting principal coefficients into the known expression for $\thet_\lambda$ (e.g.\ \cite[Example~3.10]{GHKK} or \cite[Example~3.8]{CGMMRSW}), we obtain ${\thet_{\lambda}=x_1^{-1}x_2(1+\hy_1+\hy_1\hy_2)}$.
Combining these expressions for $\thet_\lambda$, we have a functional equation for $\N(\hy_1,\hy_2)$ whose solution is \eqref{Nar cor eq}.
One part of filling in the details of this sketch is to make precise the notion of choosing $p$ ``very close'' to the limiting ray.
This requires a limiting construction that propagates through the whole argument.
\end{remark}

\subsection{Some theta functions in rank $2$}\label{rk2 theta sec}
This section is an extended example illustrating the simple-minded approach to computing theta functions in transposed cluster scattering diagrams of rank $2$ with principal coefficients.
Example~\ref{theta g2 ex} illustrates in particular how roots and co-roots interact in the construction.
We only compute some of the theta functions, but all of the theta functions are computed in \cite{CGMMRSW} by a very different approach, namely showing that they coincide with the greedy basis constructed in \cite{greedy}.

We continue to take $B=\begin{bsmallmatrix}0&b\\a&0\end{bsmallmatrix}$ with $a<0$ and $b>0$.
We also continue to place $\rho_1$ to point to the right of the plane and $\rho_2$ to point up both with the same length \emph{in the page}.
We saw in Section~\ref{rk2 inf sec} (in the infinite case $ab\le-4$) that $\Scat^T(B)$ has walls given by the coordinate lines and additionally has walls (which here we will call \newword{diagonal walls}) contained in the second quadrant.
The same is true when we allow $ab>-4$, except that there are no diagonal walls when $a=b=0$.

For brevity, we will say that a broken line \newword{scatters on} the walls containing its points of nonlinearity.

The first, third, and fourth quadrants are cones of $\gFan(B)$, so Theorem~\ref{clus mon thm} describes theta functions for vectors $m_1\rho_1+m_2\rho_2$ with $m_1\ge0$ and/or $m_2\le0$.
(If we allow $a=b=0$, then the second quadrant is also a cone of $\gFan(B)$, so we have taken $a<0$ and $b>0$ here.)
These theta functions are also easily computed directly.
We will focus on $\thet_{m_1\rho_1+m_2\rho_2}$ with $m_1<0$ and $m_2>0$.

We begin with two families of cases where we don't have to think carefully about roots and co-roots.
The first family arises from a bound on $m_1$ with respect to $b$.
The notation $[x]_+$ means $\max(x,0)$.

\begin{prop}\label{theta m1 b}
If $-b\le m_1\le0$ and $m_2\ge0$, then 
\begin{align*}
\thet_{m_1\rho_1+m_2\rho_2}&=\sum_{i=0}^{-m_1}\binom{-m_1}iy_1^ix_1^{m_1}x_2^{m_2+ai}(1+y_2x_1^b)^{[-m_2-ai]_+}\\
&=\sum_{i=0}^{-m_1}\binom{-m_1}{i}y_1^ix_0^{[-m_2-ai]_+}x_1^{m_1}x_2^{[m_2+ai]_+}\,, 
\end{align*}
where $x_0=\frac{1+y_2x_1^b}{x_2}$ is the cluster variable obtained by exchanging $x_2$ from the cluster $\set{x_1,x_2}$. 
\end{prop}
\begin{proof}
Every diagonal wall, as well as the horizontal wall, is orthogonal to a root with a positive $\alpha_2$-coordinate.
Consider a broken line $\gamma$ with initial monomial $x_1^{m_1}x_2^{m_2}$ that scatters first on a diagonal wall or on the horizontal wall.
This scattering will change the monomial to $c\cdot \hy_1^{n_1}\hy_2^{n_2}x_1^{m_1}x_2^{m_2}=c\cdot y_1^{n_1}y_2^{n_2}x_1^{m_1+bn_2}x_2^{m_2+an_1}$ for nonnegative integers $c$ and $n_1$ and positive integer $n_2$.
Since $-b\le m_1\le0$ and $b>0$, we have $m_1+bn_2\ge0$, so after scattering, the broken line moves to the left, and thus never reaches the first quadrant.
We see that $\gamma$ must either not scatter at all or must scatter first on the vertical wall.
Summing all of these possibilities, we obtain terms $x_1^{m_1}x_2^{m_2}(1+y_1x_2^a)^{-m_1}=\sum_{i=0}^{-m_1}\binom{-m_1}iy_1^ix_1^{m_1}x_2^{m_2+ai}$.
Each term represents a segment in a broken line, which can scatter on the horizontal line before reaching $p$ if and only if its slope is positive, or in other words $m_2+ai<0$.
In this case (and also when $m_2+ai=0$), the term $\binom{-m_1}iy_1^ix_1^{m_1}x_2^{m_2+ai}$ becomes terms $\binom{-m_1}iy_1^ix_1^{m_1}x_2^{m_2+ai}(1+y_2x_1^b)^{-m_2-ai}$, and otherwise the term is unchanged.
One easily convinces oneself that all of these sequences of scatterings can indeed be achieved by varying the line containing the initial infinite segment.
We see that $\thet_{m_1\rho_1+m_2\rho_2}=\sum_{i=0}^{-m_1}\binom{-m_1}iy_1^ix_1^{m_1}x_2^{m_2+ai}(1+y_2x_1^b)^{[-m_2-ai]_+}$.
\end{proof}

\begin{example}\label{theta m1 b ex}
Figure~\ref{theta m1 b fig} illustrates Proposition~\ref{theta m1 b} in the case where $a=-1$, $b=3$, $m_1=-3$ and $m_2=2$.
\begin{figure}
\scalebox{0.85}{
\includegraphics{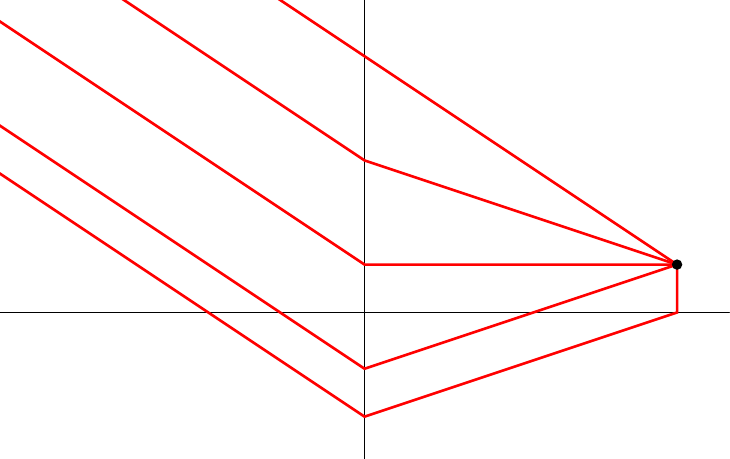}
\begin{picture}(0,0)(175,-70)
\put(0,143){$1+\hy_1=1+y_1x_2^{-1}$}
\put(-178,-10){$1+\hy_2=1+y_2x_1^3$}
\put(70,78){\textcolor{red}{$x_1^{-3}x_2^2$}}
\put(10,73){\textcolor{red}{$3y_1x_1^{-3}x_2$}}
\put(5,29){\textcolor{red}{$3y_1^2x_1^{-3}$}}
\put(-2,-10){\textcolor{red}{$y_1^3x_1^{-3}x_2^{-1}$}}
\put(150,8){\textcolor{red}{$y_1^3y_2x_2^{-1}$}}
\put(120,110){$B=\begin{bsmallmatrix}\,\,\,\,0&\,\,3\\-1&\,\,0\end{bsmallmatrix}$}
\put(150,30){$p$}
\end{picture}
}
\caption{Computing $\thet_{-3\rho_1+2\rho_2}$ for $a=-1$ and $b=3$}\label{theta m1 b fig}
\end{figure}
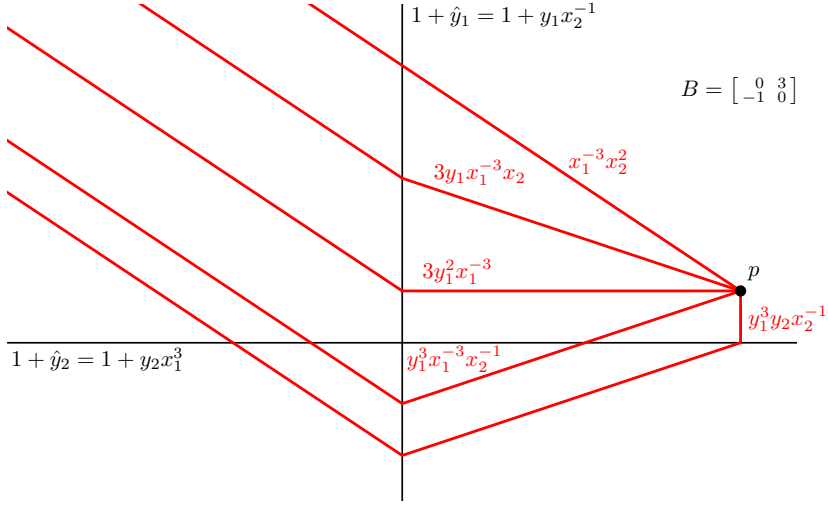
Walls are shown in black, but we have left out diagonal walls, which are irrelevant because $-b\le m_1\le0$.
Broken lines $\gamma$ are shown in red, together with the monomials $c_\gamma x^{\lambda_\gamma}y^{\beta_\gamma}$.
Scattering on the vertical wall bends the broken line in four different directions.
Three of the directions are downwards or horizontal, precluding any interaction with the horizontal wall.
One direction allows scattering (or not) on the horizontal wall.
We compute $\thet_{-3\rho_1+2\rho_2}=x_1^{-3}x_2^2+3y_1x_1^{-3}x_2+3y_1^2x_1^{-3}+y_1^3x_1^{-3}x_2^{-1}+y_1^3y_2x_2^{-1}$.
\end{example}

The second family of cases where we don't need to be careful of roots and co-roots arises from a bound on $m_2$ with respect to $a$.

\begin{prop}\label{theta m2 a}
If $m_1<-b<0$ and $0\le m_2<-a$, then $\thet_{m_1\rho_1+m_2\rho_2}$ equals
\begin{multline*}
x_1^{m_1}x_2^{m_2}+\sum_{i=1}^{-m_1}\sum_{j=0}^{m_2}\binom{-m_1-bj}i\binom{m_2}j y_1^iy_2^jx_1^{m_1+bj}x_2^{m_2+ai}(1+y_2x_1^b)^{-m_2-ai}\\
=x_1^{m_1}x_2^{m_2}+\sum_{i=1}^{-m_1}\sum_{j=0}^{m_2}\binom{-m_1-bj}i\binom{m_2}j y_1^iy_2^jx_0^{-m_2-ai}x_1^{m_1+bj},
\end{multline*}
where $x_0=\frac{1+y_2x_1^b}{x_2}$ as before.
\end{prop}
\begin{proof}
We exploit the freedom we have to choose the endpoint $p$ of the broken lines used to compute $\thet_{m_1\rho_1+m_2\rho_2}$.
Specifically, we take $p$ to be a point in the first quadrant with larger $\rho_1$-coordinate than $\rho_2$-coordinate.
We claim that no broken line with initial monomial $x_1^{m_1}x_2^{m_2}$ scatters first on a diagonal wall and then ends at $p$.
If the broken line scatters on a diagonal wall, then the monomial is changed to $c\cdot \hy_1^{n_1}\hy_2^{n_2}x_1^{m_1}x_2^{m_2}=c\cdot y_1^{n_1}y_2^{n_2}x_1^{m_1+bn_2}x_2^{m_2+an_1}$, this time for nonnegative integers $c$ and $n_2$ and positive integer $n_1$.
Since $0\le m_2<-a$, we have $m_2+an_1<0$, so the slope of the new segment is at least $1$.
Since the scattering occurs in the second quadrant and since all further scatterings only decrease the exponent on $x_2$, the broken line will never reach $p$, which was chosen to have larger $\rho_1$-coordinate than $\rho_2$-coordinate.
We conclude that each broken line contributing to $\thet_{m_1\rho_1+m_2\rho_2}$ scatters on the horizontal wall (or not), then from the vertical wall (or not), and then from the horizontal wall (or not).

The monomials that arise from scattering (or not) on the horizontal wall are $x_1^{m_1}x_2^{m_2}(1+\hy_2)^{m_2}=x_1^{m_1}x_2^{m_2}(1+y_2x_1^b)^{m_2}$, but the corresponding broken lines will never reach the first quadrant unless the exponent on $x_1$ is negative.
Thus the relevant terms are 
\[\sum_{j=0}^{\min(m_2,\lfloor\frac{-1-m_1}{b}\rfloor)}\binom{m_2}jy_2^jx_1^{m_1+bj}x_2^{m_2}.\]
When the broken lines associated to these terms scatter on the vertical wall, we obtain terms
\begin{multline*}
\sum_{j=0}^{\min(m_2,\lfloor\frac{-1-m_1}{b}\rfloor)}\binom{m_2}jy_2^jx_1^{m_1+bj}x_2^{m_2}(1+y_1x_2^a)^{-m_1-bj}\\
=\sum_{j=0}^{\min(m_2,\lfloor\frac{-1-m_1}{b}\rfloor)}\sum_{i=0}^{-m_1-bj}\binom{m_2}j\binom{-m_1-bj}i y_2^jy_1^ix_1^{m_1+bj}x_2^{m_2+ai}.
\end{multline*}
The corresponding broken lines that scatter on the vertical wall at a point below the horizontal wall will never reach the first quadrant unless the exponent on $x_2$ is negative.
This includes all broken lines that scattered first on the horizontal wall (i.e.\ all those corresponding to $j\ge1$) and also all those that first scattered on the vertical wall (i.e.\ $j=0$) and bent upwards.
(Since $p$ is on the right, those that bend upwards on the vertical wall cannot reach $p$ unless they scatter on the vertical wall at a point below the horizontal wall.)
By the hypothesis on $m_2$, the broken line bends upwards at the vertical wall if and only if $i\ge1$.
When $i\ge1$, each broken line can (for the proper choice of a line containing the initial infinite segment) scatter on the horizontal wall and reach $p$.
Thus $\thet_{m_1\rho_1+m_2\rho_2}$ has terms
\[\sum_{j=0}^{\min(m_2,\lfloor\frac{-1-m_1}{b}\rfloor)}\sum_{i=1}^{-m_1-bj}\binom{m_2}j\binom{-m_1-bj}i y_2^jy_1^ix_1^{m_1+bj}x_2^{m_2+ai}(1+y_2x_1^b)^{-m_2-ai}.\]
Since $0\le m_2<-a$, a broken line cannot scatter first on the vertical wall then continue downward or horizontally.
The only remaining possibility is that the broken line never scatters, and this gives a term $x_1^{m_1}x_2^{m_2}$.
Since $i$ is always positive, the factor $\binom{-m_1-bj}i$ is zero unless $-m_1-bj$ is also positive, so we can replace $\min(m_2,\lfloor\frac{-1-m_1}{b}\rfloor)$ with $m_2$ in the top limit of the outer sum.
Since $\binom{-m_1-bi}i=0$ when $i>-m_1-bj$, we can replace $-m_1-bj$ with $-m_1$ in the top limit of the inner sum.
\end{proof}

\begin{example}\label{theta m2 a ex}
Figure~\ref{theta m2 a fig} illustrates Proposition~\ref{theta m2 a} in the case where $a=-4$, $b=1$, $m_1=-2$ and $m_2=3$.
\begin{figure}
\scalebox{0.85}{
\includegraphics{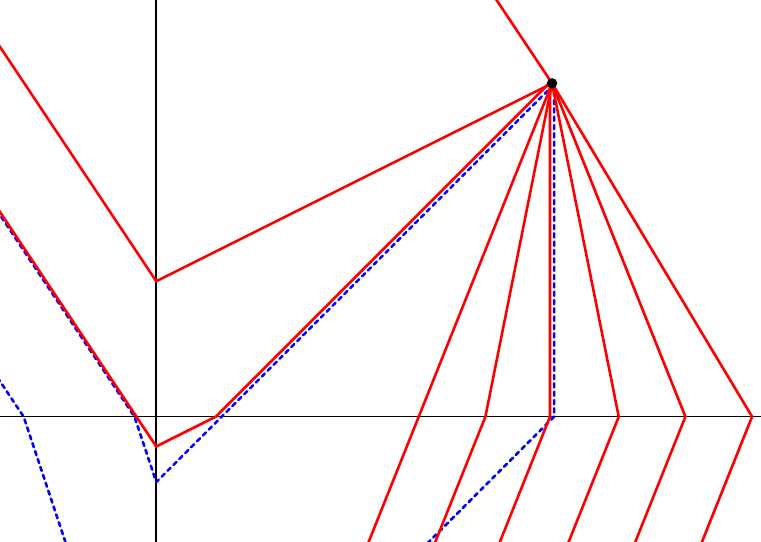}
\begin{picture}(0,0)(290,-60)
\put(0,193){$1+\hy_1=1+y_1x_2^{-4}$}
\put(-66,5){$1+\hy_2$}
\put(-62,-10){$=1+y_2x_1$}
\put(167,197){\textcolor{red}{\rotatebox{-56}{$x_1^{-2}x_2^3$}}}
\put(30,90){\textcolor{red}{\rotatebox{27}{$2y_1x_1^{-2}x_2^{-1}$}}}
\put(43,35){\textcolor{red}{\rotatebox{45}{$2y_1y_2x_1^{-1}x_2^{-1}$}}}
\put(113,15){\textcolor{red}{\rotatebox{68}{$y_1^2x_1^{-2}x_2^{-5}$}}}
\put(142,12){\textcolor{red}{\rotatebox{78}{$5y_1^2y_2x_1^{-1}x_2^{-5}$}}}
\put(171,10){\textcolor{red}{\rotatebox{90}{$10y_1^2y_2^2x_2^{-5}$}}}
\put(208,66){\textcolor{red}{\rotatebox{-78}{$10y_1^2y_2^3x_1x_2^{-5}$}}}
\put(232,56){\textcolor{red}{\rotatebox{-68}{$5y_1^2y_2^4x_1^2x_2^{-5}$}}}
\put(260,46){\textcolor{red}{\rotatebox{-59}{$y_1^2y_2^5x_1^3x_2^{-5}$}}}
\put(41,5){\textcolor{blue}{\rotatebox{45}{$3y_1y_2x_1^{-1}x_2^{-1}$}}}
\put(190,50){\textcolor{blue}{\rotatebox{-90}{$3y_1y_2^2x_2^{-1}$}}}
\put(20,160){$B=\begin{bsmallmatrix}\,\,\,\,0&\,\,1\\-4&\,\,0\end{bsmallmatrix}$}
\put(190,165){$p$}
\end{picture}
}
\caption{Computing $\thet_{-2\rho_1+3\rho_2}$ for $a=-4$ and $b=1$}\label{theta m2 a fig}
\end{figure}
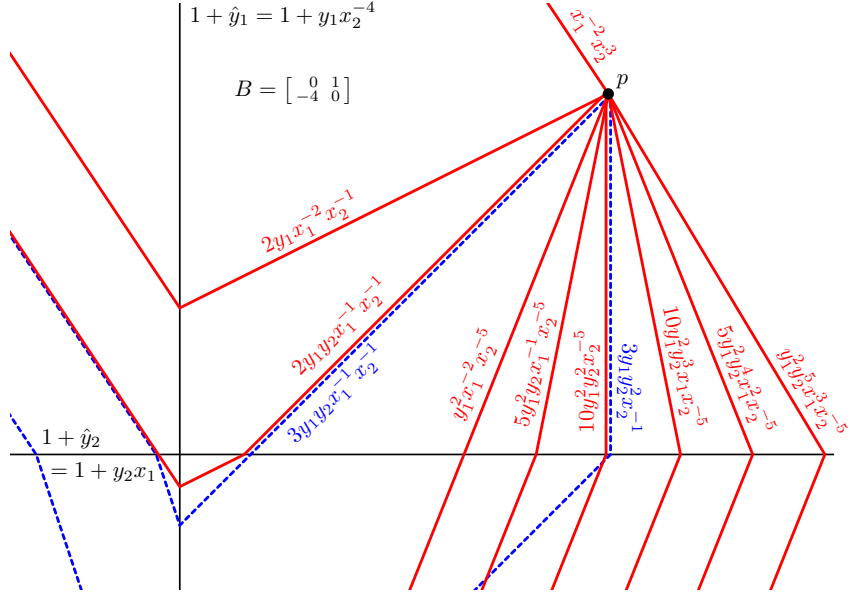
Walls are again shown in black with diagonal walls again left out because they are irrelevant.
The two broken lines that scatter first on the horizontal wall are blue and dotted.
All other broken lines are shown in red.
We show the monomials $c_\gamma x^{\lambda_\gamma}y^{\beta_\gamma}$ in the appropriate colors.
The six red broken lines and the blue dotted line that exit the bottom of the picture have all scattered (outside of the frame of the picture) on the vertical wall.

There is only one direction that broken lines can bend if they first scatter on the horizontal wall, and then they must scatter on the vertical wall, with only one possible direction.
They may then scatter, or not, on the horizontal wall.

Broken lines that first scatter on the vertical wall can do so in two different directions.
The broken lines can then scatter or not on the horizontal wall.
We compute $\thet_{-2\rho_1+3\rho_2}$ to be
\[x_1^{-2}x_2^3+3y_1y_2x_1^{-1}x_2^{-1}(1+y_2x_1)+2y_1x_1^{-2}x_2^{-1}(1+y_2x_1)+y_1^2x_1^{-2}x_2^{-5}(1+y_2x_1)^5.\]
\end{example}

Proposition~\ref{theta m1 b} is enough to compute the cluster variables in all of the rank-$2$ finite-type cases except for a single cluster variable in the case where $a=-3$ and $b=1$, namely $\thet_{-2\rho_1+3\rho_2}$.
This case also falls outside of the hypotheses of Proposition~\ref{theta m2 a}, so we work it out next as an example.
When neither Proposition~\ref{theta m1 b} nor Proposition~\ref{theta m2 a} apply, general statements become significantly harder, because we need to consider scattering on diagonal walls, but (outside of finite and affine type), we don't know all of the functions attached to the diagonal walls.

\begin{example}\label{theta g2 ex}
We compute $\thet_{-2\rho_1+3\rho_2}$ when $a=-3$ and $b=1$.
The scattering diagram is shown in Figure~\ref{theta g2 fig}.
\begin{figure}
\scalebox{0.85}{
\includegraphics{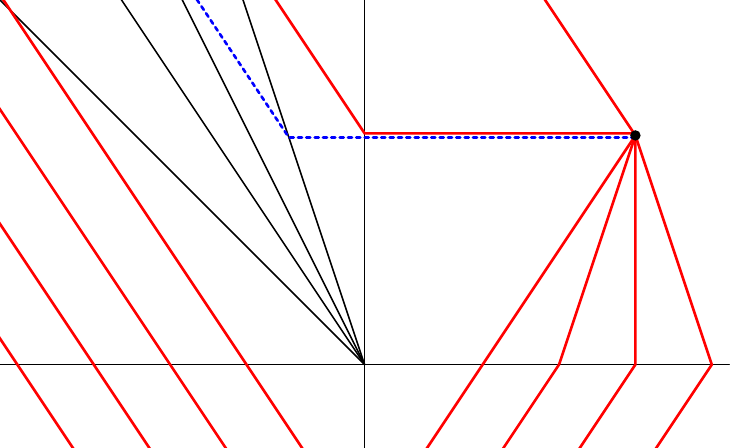}
\begin{picture}(0,0)(175,-40)
\put(-1,168){$1+\hy_1=1+y_1x_2^{-3}$}
\put(-178,-10){$1+\hy_2=1+y_2x_1$}
\put(100,158){\textcolor{red}{\rotatebox{-56}{$x_1^{-2}x_2^3$}}}
\put(40,116){\textcolor{red}{$2y_1x_1^{-2}$}}
\put(60,32){\textcolor{red}{\rotatebox{56}{$y_1^2x_1^{-2}x_2^{-3}$}}}
\put(78,4){\textcolor{red}{\rotatebox{71.5}{$3y_1^2y_2x_1^{-1}x_2^{-3}$}}}
\put(129,50){\textcolor{red}{\rotatebox{-90}{$3y_1^2y_2^2x_2^{-3}$}}}
\put(147,60){\textcolor{red}{\rotatebox{-71.5}{$y_1^2y_2^3x_1x_2^{-3}$}}}
\put(30,96){\textcolor{blue}{$3y_1y_2x_1^{-1}$}}

\put(-59,174){\rotatebox{-71.5}{$1+\hy_1\hy_2$}} 
\put(-68,133){\rotatebox{-63}{$1+\hy_1^2\hy_2^3$}} 
\put(-115,172){\rotatebox{-56}{$1+\hy_1\hy_2^2$}} 
\put(-165,165){\rotatebox{-45}{$1+\hy_1\hy_2^3$}} 

\put(20,140){$B=\begin{bsmallmatrix}\,\,\,\,0&\,\,1\\-3&\,\,0\end{bsmallmatrix}$}
\put(130,116){$p$}
\end{picture}
}
\caption{Computing $\thet_{-2\rho_1+3\rho_2}$ for $a=-3$ and $b=1$}\label{theta g2 fig}
\end{figure}
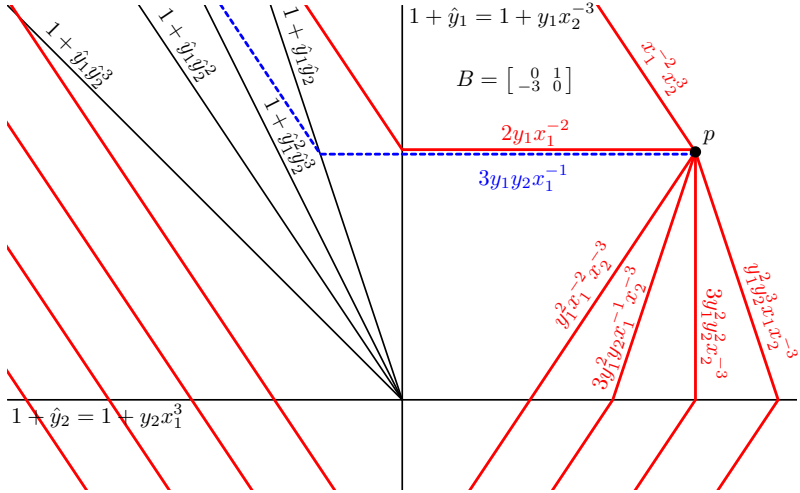
We again choose $p$ to be a point in the first quadrant whose $\rho_1$-coordinate is larger than its $\rho_2$-coordinate.
We consider first broken lines that do not scatter on diagonal walls.
These broken lines are shown in red in Figure~\ref{theta g2 fig}.
If a broken line scatters first on the horizontal wall, it can only bend one direction (because it must continue to the right), and the corresponding monomial becomes $3y_2x_1^{-1}x_2^3$.
It must then scatter on the vertical wall at a point below the horizontal wall, becoming a term of $3y_2x_1^{-1}x_2^3(1+y_1x_2^{-3})$.
However, the corresponding segments are horizontal or downwards, and thus never reach the first quadrant.
Arguing as in Proposition~\ref{theta m1 b}, we obtain terms $x_1^{-2}x_2^3+2y_1x_1^{-2}+y_1^2x_1^{-2}x_2^{-3}(1+y_2x_1)^3$.

We next consider broken lines that scatter on diagonal walls.
The functions on the diagonal walls are 
\begin{align*}
1+\hy_1\hy_2&=1+y_1y_2x_1x_2^{-3}\\
1+\hy_1^2\hy_2^3&=1+y_1y_2x_1^3x_2^{-6}\\
1+\hy_1\hy_2^2&=1+y_1y_2x_1^2x_2^{-3}\\
1+\hy_1\hy_2^3&=1+y_1y_2x_1^3x_2^{-3}
\end{align*}
There is only one of these walls on which a broken line can scatter and still move to the right, namely the wall whose function is $1+\hy_1\hy_2=1+y_1y_2x_1x_2^{-3}$.

This is the point in the calculation where one might make an error if one is not careful about roots and co-roots.
The wall with function $1+\hy_1\hy_2$ is orthogonal to the \emph{root} $\alpha_1+\alpha_2$.
As in Section~\ref{back sec}, we compute $\alpha_1=\alpha_1\ck$ and $\alpha_2=\frac13\alpha_2\ck$, so the primitive \emph{co-root} orthogonal to the wall is $3\alpha_1\ck+\alpha_2\ck$.
(That is why, as shown in Figure~\ref{theta g2 fig}, the wall consists of nonnegative multiples of $-\rho_1+3\rho_2$.)
We compute $\br{2\rho_1-3\rho_2,3\alpha_1\ck+\alpha_2\ck}=6-3=3$, so to scatter on this wall, the monomial $x_1^{-2}x_2^3$ is multiplied by a nontrivial term in $(1+\hy_1\hy_2)^3=(1+y_1y_2x_1x_2^{-3})^3$.
The only possible nontrivial term is $3y_1y_2x_1x_2^{-3}$, because every other nontrivial term will give a segment with slope $\ge1$, and the broken line will never reach $p$, whose $\rho_1$-coordinate is larger than its $\rho_2$-coordinate.
Thus, the unique broken line that scatters on a diagonal wall travels horizontally to $p$ and contributes $3y_1y_2x_1^{-1}$ to $\thet_{-2\rho_1+3\rho_2}$.
This broken line is shown blue and dotted in Figure~\ref{theta g2 fig}.
We have computed 
\begin{align}
\begin{split}
\thet_{-2\rho_1+3\rho_2}&=x_1^{-2}x_2^3+2y_1x_1^{-2}+y_1^2x_1^{-2}x_2^{-3}(1+y_2x_1)^3+3y_1y_2x_1^{-1}\\
&=x_1^{-2}x_2^3+2y_1x_1^{-2}+y_1^2x_0^3x_1^{-2}+3y_1y_2x_1^{-1}.
\end{split}
\end{align}
\end{example}

\begin{remark}\label{which thetas are cluster mons?}
When $a<0$, $b>0$ and $-ab>4$, a theta function $\thet_{m_1\rho_1+m_2\rho_2}$ is a cluster monomial if and only if $\frac{m_2}{m_1}$ \emph{does not} satisfy the inequalities for $\frac pq$ given in \eqref{pq range}.
When $a$ and/or $b$ is large in absolute value, these inequalities are roughly $a<\frac{m_2}{m_1}<\frac1{-b}$.
Under the hypotheses of  Proposition~\ref{theta m1 b}, if $-a$ is large, we can choose $m_2$ large enough so that the theta function $\thet_{m_1\rho_1+m_2\rho_2}$ is not a cluster monomial.
Under the hypotheses of Proposition~\ref{theta m2 a}, if $b$ is large, we can choose $-m_1$ large enough so that $\thet_{m_1\rho_1+m_2\rho_2}$ is not a cluster monomial.
In the affine cases, the theta functions that are not cluster monomials are all contained in a single ray.
Proposition~\ref{theta m1 b} computes one or two of these theta functions in each affine case 
($\thet_{-\rho_1+\rho_2}$ and $\thet_{-2\rho_1+2\rho_2}$ for $a=-2$ and $b=2$; 
$\thet_{-\rho_1+2\rho_2}$ for $a=-4$ and $b=1$; and  
$\thet_{-2\rho_1+\rho_2}$ and $\thet_{-4\rho_1+2\rho_2}$ for $a=-1$ and $b=4$).
Proposition~\ref{theta m2 a} does not compute (in the affine case) any theta functions that are not cluster monomials. 
\end{remark}

\section{Scattering diagrams of acyclic finite type}\label{camb sec}
When $B$ is acyclic and of finite type, the $\g$-vector fan can be constructed as a Cambrian fan.  
This fan is complete, so it is the full scattering fan $\ScatFan^T(B)$.
In this section, we review the construction of the Cambrian fan and use elements of the construction to construct scattering diagrams directly for $B$ of finite acyclic type.
Therefore, cluster monomials, and in particular, cluster variables can be constructed in terms of broken lines or path-ordered products in the Cambrian scattering diagram.

\subsection{Cambrian fans}\label{camb fan sec}
We quickly review the definition of sortable elements and Cambrian fans, skipping over a lot of the combinatorics and geometry behind the definition.
For more details, see for example \cite{sortable,camb_fan,typefree}.

We continue the notation and background from Section~\ref{back sec}.
An exchange matrix $B$ is \newword{acyclic} if there exists no cycle $i_0,\ldots,i_\ell=i_0$ of indices such that $b_{i_{k-1},i_k}>0$ for $k=1,\ldots,\ell$.
In this case, the sign information in $B$ is encoded in a choice of a \newword{Coxeter element} $c$ of $W$.
A Coxeter element is an element of $W$ that can be written as a product of the simple reflections $S$, with each simple reflection appearing exactly once.
Thus we specify a Coxeter element by giving a total order on $S$.
The exchange matrix $B$ determines a Coxeter element by multiplying the elements of $S$ in order so that $s_i$ precedes $s_j$ whenever $b_{ij}>0$.
There may be more than one way of ordering $S$ subject to this requirement, but such orders give different expressions for the same element $c$.
A simple reflection $s\in S$ is \newword{initial} in $c$ if one of these orderings has $s$ first, or equivalently if $s=s_i$ and $b_{ij}\ge0$ for $j=1,\ldots,n$.
Similarly $s$ is \newword{final} in $c$ if it is last in one of these orderings.
If $s$ is initial or final in $c$ then $scs$ is also a Coxeter element.
Furthermore, $s$ is initial in $c$ if and only if $s$ is final in $scs$.

A \newword{word} for $w\in W$ is an expression $a_1\cdots a_k$ for $w$ with $a_i\in S$ for all $i$.
The \newword{length} $\ell(w)$ of $w$ is the smallest $k$ for which a word $a_1\cdots a_k$ exists for $w$, and a \newword{reduced word} for $w\in W$ is a word $a_1\cdots a_{\ell(w)}$ for $w$.

A \newword{reflection} of $W$ is an element conjugate to some $s\in S$.
Equivalently, $t\in W$ is a reflection if and only if it has an $(n-1)$-dimensional fixed space.
In this case, the fixed space is $\beta^\perp$ for some root $\beta$, and this connection defines a bijection between reflections $t$ and positive (real) roots $\beta_t$.
An \newword{inversion} of $w\in W$ is a reflection $t$ such that $\ell(tw)<\ell(w)$.
A reflection $t$ is an inversion of $w$ if and only if $D$ and $wD$ are on opposite sides of the hyperplane $\beta_t^\perp$, for $D$ as in \eqref{D def}.
A \newword{cover reflection} of $w$ is an inversion $t$ of $w$ such that there exists $s\in S$ with $tw=ws$.
The elements $s$ that appear in this way are called \newword{(right) descents}, and the correspondence between descents of $w$ and cover reflections of $w$ is bijective.

A \newword{parabolic subgroup} of $W$ is a subgroup $W_J$ generated by some $J\subseteq S$.
For us, the most important kind of parabolic subgroup is $W_\br{s}$ for $s\in S$ and $\br{s}$ standing for $S\setminus\set{s}$.

We define the \newword{$c$-sortable elements} recursively by declaring that the identity is $c$-sortable for any $c$ in any Coxeter group $W$ and by the following two conditions for $s$ initial in $c$.
\begin{enumerate}[(i)]
\item If $w$ has a reduced word beginning with $s$ (or equivalently if $\ell(sw)<\ell(w)$), then $w$ is $c$-sortable if and only if $sw$ is $scs$-sortable.
\item If $w$ does not have a reduced word beginning with $s$ (or equivalently if $\ell(sw)>\ell(w)$), then $w$ is $c$-sortable if and only if $w$ is in $W_\br{s}$ and $w$ is $sc$-sortable as an element of $W_\br{s}$.
\end{enumerate}
These conditions decide the $c$-sortability of an element $w$ by induction on $\ell(w)$ and on $n$.
They make sense in particular because, when $s$ is initial in $c$, $scs$ is a Coxeter element of $W$ and $sc$ is a Coxeter element of $W_\br{s}$.
The notion of $c$-sortability is well-defined in light of the usual non-recursive definition.  

For each $c$-sortable element $v$, we define a set $C_c(v)\subset\Phi$ with $|C_c(v)|=n$.
This can be defined in terms of ``skips'' in a special ``$c$-sorting word'' for $v$, as explained in \cite[Section~5]{typefree}, but here we give the simple recursive definition:
If $v$ is the identity, then $C_c(v)=\set{\alpha_1,\ldots,\alpha_n}$ (the set of simple roots), and for $s=s_i$ initial in $c$, 
\begin{equation}\label{Cc def}
C_c(v)=\begin{cases}
sC_{scs}(sv)&\text{if }\ell(sv)<\ell(v),\text{or}\\
C_{sc}(v)\cup\set{\alpha_i}&\text{if }\ell(sv)>\ell(v).
\end{cases}
\end{equation}
Furthermore, define 
\begin{equation}
\Cone_c(v)=\bigcap_{\beta\in C_c(v)}\set{p\in V^*:\br{p,\beta}\ge0}.
\end{equation}
The cone $\Cone_c(v)$ contains the cone $vD$ for $D$ as in \eqref{D def}, as an immediate consequence of \cite[Theorem~6.3]{typefree}.
In particular, $\Cone_c(v)$ is full-dimensional.
It is also simplicial.
No hyperplane $\alpha_i^\perp$ intersects the interior of any cone $\Cone_c(v)$.
(This is the concatenation of \cite[Proposition~4.26]{framework} with \cite[Theorem~5.12]{framework} or, in the finite-type case, is a special case of \cite[Proposition~9.5]{typefree}.)

For each $J\subseteq S$, let $V_J$ be the subspace of $V$ spanned by $\set{\alpha_i:s_i\in J}$ and identify $(V_J)^*$ in the natural way with the subspace of $V^*$ spanned by $\set{\rho_i:s_i\in J}$.
Then \eqref{Cc def} implies the following recursion for $s=s_i$ initial in $c$.
\begin{equation}\label{cone recur}
\Cone_c(v)=\begin{cases}
s\Cone_{scs}(sv)&\text{if }\ell(sv)<\ell(v),\text{or}\\
\Span_{\reals_{\ge0}}(\Cone_{sc}(v)\cup\set{\rho_i})&\text{if }\ell(sv)>\ell(v).
\end{cases}
\end{equation}
Here $\Cone_{sc}(v)$ lives in the subspace $(V_\br{s})^*$ of $V^*$.
(This subspace is $\alpha_i^\perp$ for $s=s_i$.)
The notation $\Span_{\reals_{\ge0}}$ means nonnegative linear span.

The \newword{$c$-Cambrian fan} $\F_c$ is the collection consisting of the cones $\Cone_c(v)$ and their faces, for all $c$-sortable elements $v$.
The following theorem is the concatenation of \cite[Corollary~5.15]{framework} with Theorem~\ref{clus mon thm}.
Versions of \cite[Corollary~5.15]{framework} appear as \cite[Theorem~9.1]{typefree}, \cite[Theorem~1.10]{YZ} (see \cite[Remark~6.1]{YZ}), and by concatenating results of \cite{con_app,sort_camb}.

\begin{theorem}\label{camb fan thm}
If $B$ is acyclic with associated Coxeter element $c$, then $\F_c$ is a simplicial fan, and in particular a subfan of $\gFan(B)$.
\end{theorem}

\subsection{Cambrian scattering diagrams}
Continuing the notation from above, we now use the sortable/Cambrian machinery to directly construct a (transposed) scattering diagram equivalent to $\ScatFan^T(B)$.

The walls of the scattering diagram will be the codimension-$1$ faces of the $c$-Cambrian fan.
Each such face is the intersection of two maximal faces $\Cone_c(v)$ and $\Cone_c(v')$, and thus is contained in $\beta^\perp$ for some root with $\beta\in C_c(v)$ and $-\beta\in C_c(v')$.
Given a $c$-sortable element $v$, by \cite[Proposition~5.2]{typefree}, the set of negative roots in $C_c(v)$ equals $\set{-\beta_t:t\in\cov(v)}$, where $\cov(v)$ is the set of cover reflections of $v$.
Thus we can list each codimension-$1$ face of the $c$-Cambrian fan exactly once by running through all $c$-sortable elements $v$ and all cover reflections of each $v$.
Accordingly, we define the \newword{$c$-Cambrian scattering diagram} to be the set 
\begin{equation}
\CambScat(A,c)=\set{(\Cone_c(v)\cap\beta_t^\perp,1+\hy^{\beta_t}):t\in\cov(v)}.
\end{equation}
The notation emphasizes that the information in $B$ is equivalent to the information in the Cartan matrix $A$ and Coxeter element $c$.
This is a scattering diagram in the transposed sense of Section~\ref{prin trans sec}.

\begin{example}\label{scat rk2}
Figure~\ref{camb scat rk2} shows $\CambScat(A,c)$ for all cases with $n=2$, up to Proposition~\ref{scat antip} and the symmetry of swapping the indices $1$ and $2$.
For the associated root systems, see Figure~\ref{rk2roots}.
\end{example}


\begin{figure}
\renewcommand*{\arraystretch}{1.4}
\setlength{\doublerulesep}{1pt}
\medskip
\begin{tabular}{ll|||ll}
\hline
\!\!\!\!\!\raisebox{66pt}{\small$\begin{array}{l}
B=\begin{bsmallmatrix}0&0\\0&0\end{bsmallmatrix}\\
A=\begin{bsmallmatrix}2&0\\0&2\end{bsmallmatrix}\\
c=s_1s_2\\
\hy_1=y_1\\
\hy_2=y_2
\end{array}$}\!\!\!\!\!\!&
\scalebox{0.86}{\includegraphics{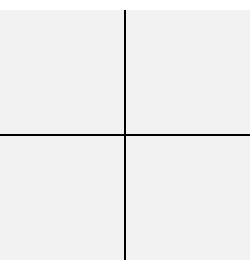}
\begin{picture}(0,0)(60,-60)
\put(28,4){$1+\hy_2$}
\put(-61,4){$1+\hy_2$}
\put(1,51){$1+\hy_1$}
\put(1,-56){$1+\hy_1$}
\end{picture}}
&
\!\!\!\!\raisebox{66pt}{\small$\begin{array}{l}
B=\begin{bsmallmatrix}\,\,\,\,0&1\\-1&0\end{bsmallmatrix}\\
A=\begin{bsmallmatrix}\,\,\,\,2&-1\\-1&\,\,\,\,2\end{bsmallmatrix}\\
c=s_1s_2\\
\hy_1=y_1x_2^{-1}\\
\hy_2=y_2x_1
\end{array}$}\!\!\!\!\!\!&
\scalebox{0.86}{\includegraphics{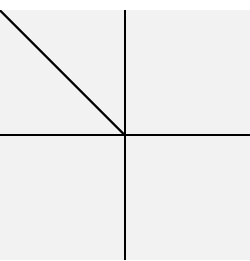}
\begin{picture}(0,0)(60,-60)
\put(29,4){$1+\hy_2$}
\put(-63,4){$1+\hy_2$}
\put(-53,56){\rotatebox{-45}{$1+\hy_1\hy_2$}}
\put(0,52){$1+\hy_1$}
\put(0,-56){$1+\hy_1$}
\end{picture}}\!\!\!
\\\hline
\!\!\!\!\!\!\raisebox{66pt}{\small$\begin{array}{l}
B=\begin{bsmallmatrix}\,\,\,\,0&1\\-2&0\end{bsmallmatrix}\\
A=\begin{bsmallmatrix}\,\,\,\,2&-1\\-2&\,\,\,\,2\end{bsmallmatrix}\\
c=s_1s_2\\
\hy_1=y_1x_2^{-2}\\
\hy_2=y_2x_1
\end{array}$}\!\!\!\!\!\!&
\scalebox{0.86}{\includegraphics{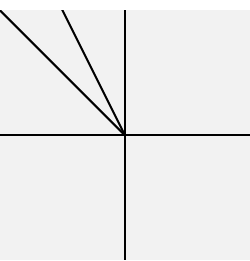}
\begin{picture}(0,0)(60,-60)
\put(29,4){$1+\hy_2$}
\put(-63,-10){$1+\hy_2$}
\put(-67,48){\rotatebox{-45}{$1+\hy_1\hy_2^2$}}
\put(-30,59){\rotatebox{-63.4}{$1+\hy_1\hy_2$}}
\put(0,52){$1+\hy_1$}
\put(0,-56){$1+\hy_1$}
\end{picture}}
&\!\!\!\!\raisebox{66pt}{\small$\begin{array}{l}
B=\begin{bsmallmatrix}\,\,\,\,0&1\\-3&0\end{bsmallmatrix}\\
A=\begin{bsmallmatrix}\,\,\,\,2&-1\\-3&\,\,\,\,2\end{bsmallmatrix}\\
c=s_1s_2\\
\hy_1=y_1x_2^{-3}\\
\hy_2=y_2x_1
\end{array}$}\!\!\!\!\!\!&
\scalebox{0.607}{\includegraphics{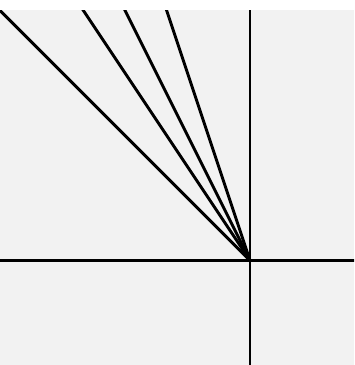}
\begin{picture}(0,0)(50,-50)
\put(20,4.5){$1+\hy_2$}
\put(-120,4.5){$1+\hy_2$}
\put(0,112){$1+\hy_1$}
\put(0,-46){$1+\hy_1$}
\put(-40,118){\rotatebox{-71.5}{$1+\hy_1\hy_2$}} 
\put(-60,118){\rotatebox{-63}{\small$1+\hy_1^2\hy_2^3$}} 
\put(-96,118){\rotatebox{-56}{$1+\hy_1\hy_2^2$}} 
\put(-125,106){\rotatebox{-45}{$1+\hy_1\hy_2^3$}} 
\end{picture}}\!\!\!
\\\hline
\end{tabular}
\caption{Cambrian scattering diagrams with $n=2$}
\label{camb scat rk2}
\end{figure}
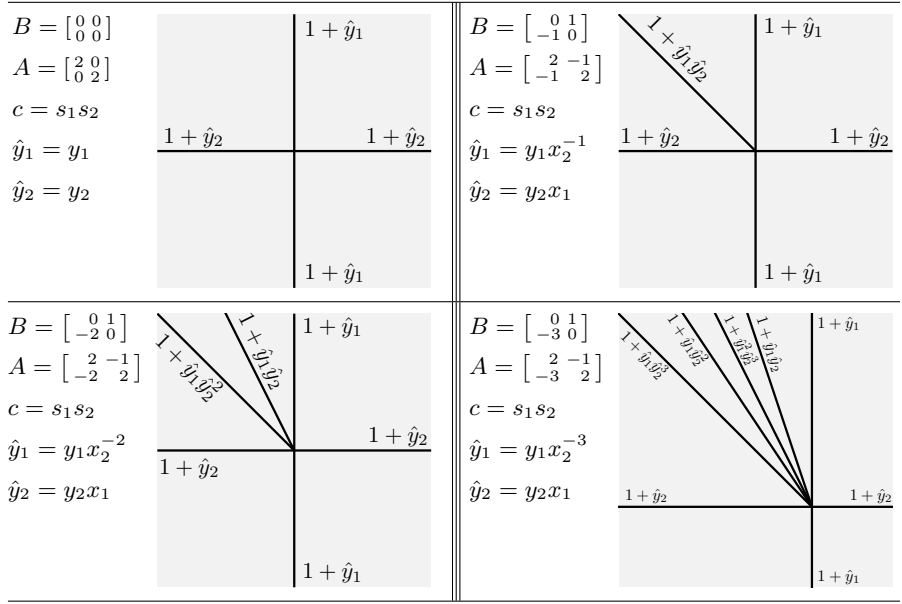

%

Our goal is to prove the following theorem.

\begin{theorem}\label{camb scat thm}
If $B$ is acyclic of finite type with associated Cartan matrix $A$ and Coxeter element $c$, then $\Scat^T(B)$ is equivalent to $\CambScat(A,c)$.
\end{theorem}

Since $\CambScat(A,c)$ is a finite collection of walls, it is a scattering diagram.
Since none of the maximal cones of $\F_c$ intersect any hyperplane $\alpha_i^\perp$ in their interior, and since $\F_c$ is a finite, complete fan, each $\alpha_i^\perp$ is the union of finitely many walls of $\CambScat(A,c)$, each with function $1+\hy_i$.
Up to equivalence, these walls can be replaced by a single wall $(\alpha_i^\perp,1+\hy_i)$.
To complete the proof of Theorem~\ref{camb scat thm}, we show that every other wall is outgoing and that $\CambScat(A,c)$ is consistent, in Propositions~\ref{camb out} and~\ref{camb consist}, below.

First, we need some more background on the form $\omega$ in the acyclic case.
In this case, we write $\omega_c$ for $\omega$, where $c$ is the Coxeter element associated to $B$.
The point is that we can fix the Cartan matrix $A$ while changing $c$,
as in the following lemma, which is \cite[Lemma~3.8]{typefree}.

\begin{lemma}\label{omega s}
If $s$ is initial in $c$, then $\omega_c(\phi\ck,\beta)=\omega_{scs}(s\phi\ck,s\beta)$ for $\phi\ck\in Q\ck$ and $\beta\in Q$.
\end{lemma}

We rewrite Lemma~\ref{omega s} in terms of the dual action on $V^*$ as follows.

\begin{lemma}\label{omega s dual}
If $s$ is initial in $c$, then for any $\beta\in Q$, the action of $s$ sends $\omega_c(\,\cdot\,,\beta)$ to $\omega_{scs}(\,\cdot\,,s\beta)$.
\end{lemma}

We now prove part of Theorem~\ref{camb scat thm}.
\begin{proposition}\label{camb out}
Every wall of $\CambScat(A,c)$ not contained in some $\alpha_i^\perp$ is outgoing.
\end{proposition}
This proposition is closely related to the notion of \newword{$c$-alignment} in \cite[Theorem~4.1]{sortable}---see also \cite[Theorem~4.3]{typefree}---but here, we give a simple recursive proof.
\begin{proof}
Each wall is $(\d,f_\d)$ with $\d=\Cone_c(v)\cap\beta_t^\perp$ for $t\in\cov(v)$ and thus $-\beta_t\in C_c(v)$, but there is another cone $\Cone_c(v')$ sharing that wall as a facet and having $\beta_t\in C_c(v')$.
Thus it is enough to show that for every $c$-sortable $v$ and every \emph{positive} root $\beta\in C_c(v)\setminus\set{\alpha_1,\ldots,\alpha_n}$, the vector $\omega_c(\,\cdot\,,\beta)$ is not in $\Cone_c(v)$.
We argue by induction on $\ell(v)$ and on $n$.

Suppose $v$ is $c$-sortable and $\beta$ is a positive root in $C_c(v)\setminus\set{\alpha_1,\ldots,\alpha_n}$.
Let $s_i$ be initial in $c$.
First consider the case where $\ell(s_iv)<\ell(v)$.
If $s_i\beta$ is not a positive root, then $\beta=\alpha_i$, because $\pm\alpha_i$ are the only roots whose sign changes under the action of $s_i$.
This is ruled out by hypothesis, so $s_i\beta$ is positive.
We have $s_i\beta\in C_{s_ics_i}(s_iv)$ by \eqref{Cc def}, and we will show that $\omega_{s_ics_i}(\,\cdot\,,s_i\beta)\not\in\Cone_{s_ics_i}(s_iv)$.

If $s_i\beta\not\in\set{\alpha_1,\ldots,\alpha_n}$, then by induction on $\ell(v)$, the vector $\omega_{s_ics_i}(\,\cdot\,,s_i\beta)$ is not in $\Cone_{s_ics_i}(s_iv)$.
If on the other hand $s_i\beta=\alpha_j$ for some $j$, then $i\neq j$ because otherwise $\beta=-\alpha_i$.
We have $\omega_{s_ics_i}(\,\cdot\,,s_i\beta)=\sum_{k=1}^nb'_{kj}\rho_k$, where $B'=[b'_{ij}]$ is the exchange matrix defined by $A$ and $s_ics_i$.
The Cartan matrix entry $a_{ij}$ is nonzero, because otherwise $\beta=s_i\alpha_j=\alpha_j$.
Since $s_i$ is final in $s_ics_i$, we have $b'_{ij}<0$, and we see that $\omega_{s_ics_i}(\alpha_i\ck,s_i\beta)<0$.
But $\ell(s_iv)<\ell(v)$, so $s_i$ is an inversion of $v$, so that $\alpha_i^\perp$ separates $D$ from $vD$.
Therefore $\alpha_i^\perp$ does not separate $D$ from $s_ivD$.
Since $s_ivD\subseteq\Cone_{s_ics_i}(s_iv)$, which does not cross $\alpha_i^\perp$, we see that $\Cone_{s_ics_i}(s_iv)$ is contained in $\set{p\in V^*:\br{p,\alpha_i\ck}\ge0}$.
However, we have calculated that $\omega_{s_ics_i}(\,\cdot\,,s_i\beta)\not\in\set{p\in V^*:\br{p,\alpha_i\ck}\ge0}$.
Thus in this case also $\omega_{s_ics_i}(\,\cdot\,,s_i\beta)\not\in\Cone_{s_ics_i}(s_iv)$.

In either case, $\Cone_c(v)=s_i\Cone_{s_ics_i}(s_iv)$ by \eqref{cone recur}.
But $\omega_c(\,\cdot\,,\beta)$ is equal to $s_i\omega_{s_ics_i}(\,\cdot\,,s_i\beta)$ by Lemma~\ref{omega s dual}, so $\omega_c(\,\cdot\,,\beta)$ is not in $\Cone_c(v)$.

It remains to check the case where $\ell(s_iv)>\ell(v)$.
In this case, $C_c(v)=C_{s_ic}(v)\cup\set{\alpha_i}$, so we need only consider positive roots $\beta\in C_{s_ic}(v)\setminus\set{\alpha_1,\ldots,\alpha_n}$.
By induction on $n$, the vector $\omega_{s_ic}(\,\cdot\,,\beta)\in(V_\br{s_i})^*$ is not in $\Cone_{s_ic}(v)$.
Therefore, there exists $\phi\in C_{s_ic}(v)$ such that $\omega_{s_ic}(\phi\ck,\beta)<0$.
In particular, $\phi$ is in the span of $\set{\alpha_1,\ldots,\alpha_n}\setminus\set{\alpha_i}$.
Since $\omega_c(\,\cdot\,,\beta)$ and $\omega_{sc}(\,\cdot\,,\beta)$ differ by a multiple of $\rho_i$, we see that $\omega_{c}(\phi\ck,\beta)<0$, so that $\omega_c(\,\cdot\,,\beta)\not\in\Cone_c(v)$.
\end{proof}

The proof of the last piece of Theorem~\ref{camb scat thm} uses a result \cite[Theorem~9.8]{typefree} on the local structure of the $c$-Cambrian fan $\F_c$ to reduce to the case where $n=2$.
We now explain the case of \cite[Theorem~9.8]{typefree} that we need.

Suppose $G$ is a face of $\F_c$ of codimension $2$.
If $p$ is in the interior of $D$ and $q$ is in the relative interior of $G$, then as $\ep>0$ approaches $0$, the vector $q-\ep p$ remains in the interior of some maximal cone $\Cone_c(v)$ of $\F_c$.
We call $v$ the \newword{$c$-sortable element above $G$}.
The cone $G$ is a face of $\Cone_c(v)$, so $G=\Cone_c(v)\cap\beta_1^\perp\cap\beta_2^\perp$ for $\beta_1,\beta_2\in C_c(v)$.
By our choice of $v$, both $\beta_1$ and $\beta_2$ are negative, so $\beta_1=-\beta_{t_1}$ and $\beta_2=-\beta_{t_2}$ for $t_1,t_2\in\cov(v)$.
Let $a_1$ and $a_2$ be the elements of $S$ such that $t_1=va_1v^{-1}$ and $t_2=va_2v^{-1}$.
Write $\sigma_1$ and $\sigma_2$ for the fundamental weights dual to $a_1$ and $a_2$.

The cone $vD$ is contained in $\Cone_c(v)$, and since $t_1,t_2\in\cov(v)$ and $G=\Cone_c(v)\cap\beta_1^\perp\cap\beta_2^\perp$, the intersection $vD\cap G$ is a codimension-$2$ face of $vD$.
Taking $q$ in the relative interior of $vD\cap G$ and $p$ as before, as $\ep>0$ approaches $0$, the vector $q+\ep p$ remains in the interior of some cone $wD$ for $w\in W$.
We call $w$ the \newword{element below $G$}.
We may assume that we numbered $\beta_1$ and $\beta_2$ (and thus $a_1$ and $a_2$) so that $\omega_c(\beta_{wa_1w^{-1}},\beta_{wa_2w^{-1}})\ge0$.
Note that $\set{\beta_{wa_1w^{-1}},\beta_{wa_2w^{-1}}}=\set{\beta_1,\beta_2}$, but it is possible that $\beta_{wa_1w^{-1}}=\beta_2$.

Given any cone $F$ and any $p\in F$, the \newword{linearization $\Lin_p(F)$ of $F$ at $x$} is the set of vectors $q\in V^*$ such that $p+\ep q$ is in $F$ for sufficiently small $\ep>0$.
For any cone $G$ in $\F_c$, we fix $p$ in the relative interior of $G$ and define the \newword{star of $G$ in $\F_c$} to be the collection $\Star_G(\F_c)$ consisting of cones $\Lin_p(F)$ such that $F$ is a cone in $\F_c$ containing $G$.
This collection is independent of the choice of $p$ and is a fan in $V^*$ (and indeed a complete fan, since we are in the finite-type case where $\F_c$ is complete).
In the special case where $G$ has codimension $2$ and continuing the notation from above, \cite[Theorem~9.8]{typefree} says that 
\begin{equation}\label{9.8}
\Star_G(\F_c)=w\bigl(\F_{a_1a_2}\oplus\Span_\reals(\set{\rho_1,\ldots,\rho_n}\setminus\set{\sigma_1,\sigma_2})\bigr).
\end{equation}
Here $a_1a_2$ is a Coxeter element of the parabolic subgroup $W_{\set{a_1,a_2}}$ and $\F_{a_1a_2}$ is the $a_1a_2$-Cambrian fan constructed in $(V_{\set{a_1,a_2}})^*$, which we have identified with the subspace of $V^*$ spanned by $\set{\sigma_1,\sigma_2}$.
The fan $\F_{a_1a_2}\oplus\Span_\reals(\set{\rho_1,\ldots,\rho_n}\setminus\set{\sigma_1,\sigma_2}$ is the collection of cones $F\oplus\Span_\reals(\set{\rho_1,\ldots,\rho_n}\setminus\set{\sigma_1,\sigma_2})$ such that $F$ is a cone in $\F_{a_1a_2}$.

\begin{prop}\label{camb consist}
$\CambScat(A,c)$ is consistent.
\end{prop}
\begin{proof}  
For the purpose of defining path-ordered products, generic paths in $V^*\setminus\Supp(\CambScat(A,c))$ are indistinguishable from sequences $F_0,\ldots,F_k$ of maximal cones in $\F_c$ with $F_{i-1}$ and $F_i$ adjacent for $i=1,\ldots,k$.
We write $\p_{F_0,\ldots,F_k;c}$ for the path-ordered product $\p_{\gamma,\CambScat(A,c)}$ such that $\gamma$ is a generic path starting in the interior of $F_0$, passing in succession to the interiors of the $F_i$, and finally ending in the interior of $F_k$.
Consistency of $\CambScat(A,c)$ means showing that $\p_{F_0,\ldots,F_k;c}$ is the identity whenever $F_0=F_k$.
Checking this is easily reduced to the case where $F_0,\ldots,F_k$ visits exactly once each maximal cone that contains some fixed codimension-$2$ cone $G$.

We continue the notation that was used above to define the $c$-sortable element above $G$, the element below $G$, etc.
Since every wall crossed by $F_0,\ldots,F_k$ contains $G$, the path-ordered product $\p_{F_0,\ldots,F_k;c}$ acts as the identity on $x^\lambda$ when $\lambda$ is in the linear span of $G$.
The span of $\set{w\sigma_1,w\sigma_2}$ is complementary to the linear span of~$G$, so it is enough to check that $\p_{F_0,\ldots,F_k;c}$ acts as the identity on $x^{w\sigma_1}$ and $x^{w\sigma_2}$.

We write $f_1,\ldots,f_k$ for the sequence of wall functions encountered in evaluating $\p_{F_0,\ldots,F_k;c}$.
Let $f'_1,\ldots,f'_\ell$ be the sequence of wall functions encountered along a cycle about the origin in $\CambScat(A',a_1a_2)$, where $A'$ is the restriction of $A$ to the rows and columns of $a_1$ and $a_2$.
Then \eqref{9.8} implies that $\ell=k$ and (when we choose the right starting point for the cycle in $\CambScat(A',a_1a_2)$) that if $f'_i=1+\hy^\beta$ then $f_i=1+\hy^{w\beta}$.
We can check consistency in each rank-$2$ case (shown in Figure~\ref{camb scat rk2}) or compare with results of \cite{CGMMRSW} to see that $\CambScat(A',a_1a_2)$ is consistent.
In particular, the path-ordered product for the cycle in $\CambScat(A',a_1a_2)$ fixes both $x^{\sigma_1}$ and $x^{\sigma_2}$.
The path-ordered product on the corresponding cycle $F'_0,\ldots,F'_k$ in $\CambScat(A,c)$ differs only by replacing the functions $\hy$ for $A'$ with the corresponding functions $\hy$ for $A$.
The effect is to multiply each $\hy_i$ by a Laurent monomial $x^\lambda$ with $\lambda\in\Span_\integers(\set{\rho_1,\ldots,\rho_n}\setminus\set{\sigma_1,\sigma_2})$, so the replacement commutes with the wall-crossing automorphisms encountered on this cycle.
We see that $\p_{F'_0,\ldots,F'_k;c}$ acts as the identity.
Now, by \eqref{9.8}, using Proposition~\ref{just x's} to avoid worrying about values of $\omega_c$, and keeping in mind the definition of the dual action, we see that $\p_{F_0,\ldots,F_k;c}$ acts as the identity as well.
\end{proof}

This completes the proof of Theorem~\ref{camb scat thm}

\begin{remark}\label{g method}
When $B$ is acyclic and of finite type, the $c$-Cambrian fan coincides with $\gFan(B)$.
(This was conjectured and partially proved in \cite[Section~10]{camb_fan} and proved first in \cite{YZ} and then as \cite[Corollary~5.16]{framework}.)
By Theorem~\ref{clus mon thm} (which follows from results of \cite{GHKK}), in finite type, the transposed scattering diagram coincides with $\gFan(B)$.
In particular, the $c$-Cambrian fan and the transposed scattering diagram coincide, and Theorem~\ref{camb scat thm} follows immediately by Theorem~\ref{clus easy root}.
Here we have chosen to make the connection between Cambrian fans and scattering diagrams directly, rather than connecting them indirectly via $\g$-vector fans.

However, using the $\g$-vector fan may be useful in not-necessarily-acyclic finite type.
In that setting as well, if one constructs $\gFan(B)$, then one obtains the transposed cluster scattering diagram easily by putting functions on the codimension-$1$ faces according to Theorem~\ref{clus easy root}.
For example, the main result of \cite{cyclica} is a construction of the $\g$-vector fan for the oriented cycle (a non-acyclic exchange matrix of finite type D) using a root system of affine type A.
In that case, putting the function $1+\hy^\beta$ on each codimension-$1$ face of the fan normal to the root $\beta$ yields a scattering diagram.

Alternately, still in not-necessarily-acyclic finite type, if one constructs what one thinks is the $\g$-vector fan, then one can prove that it is indeed the $\g$-vector fan by putting functions on the codimension-$1$ faces according to Theorem~\ref{clus easy root}, then showing consistency, checking that each $\alpha_i^\perp$ is covered by codimension-$1$ faces, and checking that all other walls are outgoing.
(Following this idea in acyclic finite type, our proof of Theorem~\ref{camb scat thm} amounts to a new proof that the $c$-Cambrian fan is the $\g$-vector fan in finite type.)
\end{remark}

\begin{remark}\label{comput}
In light of Theorem~\ref{camb scat thm}, we can compute cluster variables (or more generally cluster monomials) in acyclic finite type as explained in Theorems~\ref{clus mon thm} and~\ref{clus mon pop root}.
Thus the cluster monomial with $\g$-vector $\lambda$ (encoded as usual as an element of the weight lattice) is $\thet_\lambda$, as defined in Section~\ref{trans clus sec}.
Alternately, this cluster monomial is $\p_{q,p,\CambScat(B)}(x^\lambda)$, where $p$ is any point in the interior of $D$ and $q$ is any point in a maximal cone $F$ of $\gFan(B)$ with $\lambda\in F$.
We briefly discuss the computation of $\p_{q,p,\CambScat(B)}(x^\lambda)$.
The discussion requires some background on the weak order and $c$-sortable elements which can be found in any of the references on sortable elements or Cambrian fans already mentioned.

We compute $\p_{q,p,\CambScat(B)}(x^\lambda)$ using a sequence of maximal cones of $\F_c$, as in the proof of Theorem~\ref{camb scat thm}.
Choose some $c$-sortable element $v_0$ such that $\lambda\in\Cone_c(v_0)$.
We want a sequence $v_0,\ldots,v_k$ such that $\Cone_c(v_{i-1})$ and $\Cone_c(v_i)$ are adjacent for each $i=1,\ldots,k$ and such that $\Cone_c(v_k)=D$, or equivalently $v_k$ is the identity.
Ideally, we would choose $v_0$ and the rest of the sequence to minimize~$k$.
However, it is easier to choose $v_0$ to be the unique minimal (in weak order) $c$-sortable element $v_0$ such that $\lambda\in\Cone_c(v_0)$, require that $\ell(v_i)<\ell(v_{i-1})$ for all $i$ and \emph{then} minimize $k$.
A heuristic to make $k$ small is as follows:
For each $v_{i-1}$, by \cite[Proposition~3.11]{typefree}, there exists $t\in\cov(v_{i-1})$ such that $\omega_c(\beta_t,\beta_{t'})\ge0$ for all $t'\in\cov(v_{i-1})$.
Define $v_i$ to be the unique maximal $c$-sortable element that is below $tv_{i-1}$ in the weak order.
Now \cite[Proposition~3.11]{typefree} suggests that we are minimizing $k$, subject to $\ell(v_i)<\ell(v_{i-1})$ for all $i$.
\end{remark}

\subsection{Shards}
Our construction of $\CambScat(A,c)$ made each codimension-$1$ face of the $c$-Cambrian fan into a wall.
Typically, there is more than one codimension-$1$ face orthogonal to each root, so one expects that one could combine some of the codimension-$1$ faces into a smaller number of walls.
In this section, we show how to construct a scattering diagram equivalent to $\CambScat(A,c)$---and thus equivalent to $\Scat^T(B)$---with exactly one wall orthogonal to each root.

The walls in this scattering diagrams will be \newword{shards}---or, more specifically, the shards not removed by the $c$-Cambrian congruence $\Theta_c$.
Shards play a role in the combinatorics, geometry, and lattice theory of the weak order on a finite Coxeter group, and play a key role in papers including \cite{bancroft,BCHT,petersen,hyperplane,hplanedim,congruence,sort_camb,shardint}.
The $c$-Cambrian congruence is the key player in a lattice-theoretic approach \cite{cambrian,sort_camb} to sortable elements and Cambrian fans.
Here, we will not go into details about shards and $c$-Cambrian congruences.
Rather, we simply quote results that make the connection.

First, by definition---see for example \cite[Section~3]{shardint}---shards are \mbox{codimension-$1$} closed convex cones defined by hyperplanes $\beta^\perp$ for $\beta\in\Phi$.
(In \cite{hyperplane,hplanedim}, the shards were relatively open, but in later references, closures were taken as part of the definition.)
Second, again by definition, for each hyperplane $\beta^\perp$, the shards contained in $\beta^\perp$ exactly cover $\beta^\perp$ with non-intersecting relative interiors.
Third, in each hyperplane $\beta^\perp$, the union of all codimension-$1$ faces of the $c$-Cambrian fan that are orthogonal to $\beta$ is a shard \cite[Proposition~8.15]{shardint}.
(Which shard this is depends on~$c$.)
We call this shard ``the shard in $\beta^\perp$ not removed by $\Theta_c$'' for consistency with other references on shards and Cambrian congruences, but since we are not giving any definitions about shards and Cambrian congruences, we can safely think of ``the shard in $\beta^\perp$ not removed by $\Theta_c$'' merely as a very long proper name, and refer to it by a shorter proper name $\Sigma_c(\beta)$.

Combining this information on shards and Cambrian fans with Theorem~\ref{camb scat thm}, we have the following corollary, which shows in particular that $\Scat^T(B)$ is equivalent to a scattering diagram with exactly one wall orthogonal to each positive root.

\begin{cor}\label{shard scat}
If $B$ is acyclic of finite type with associated Cartan matrix $A$ and Coxeter element $c$, then $\Scat^T(B)$ is equivalent to $\set{(\Sigma_c(\beta),1+\hy^\beta):\beta\in\Phi_+}$.
\end{cor}

Corollary~\ref{shard scat} makes possible explicit constructions or computations of scattering diagrams in finite acyclic type.
Explicit inequalities defining shards can be derived using \cite[Lemma~3.7]{congruence}.
Details in type A are found in \cite{bancroft,petersen}.

%

Each hyperplane $\beta^\perp$ (for $\beta\in\Phi_+$) typically has many shards, and exactly one of them is a wall in the scattering diagram $\set{(\Sigma_c(\beta),1+\hy^\beta):\beta\in\Phi_+}$.
If the decomposition of $\beta^\perp$ into shards is known, we can pick out the shard $\Sigma_c(\beta)$ using the bilinear form $\omega_c$.
Call a shard (or a wall) $\Sigma$ \newword{gregarious} if the vector $-\omega_c(\,\cdot\,,\beta)$ is in the relative interior of $\Sigma$.
Since the relative interiors of shards in $\beta^\perp$ are disjoint, there can be at most one gregarious shard in~$\beta^\perp$.

\begin{remark}\label{greg remark}
The term ``gregarious'' alludes to the notion of an outgoing wall in a scattering diagram.
Recall that a wall $(\d,f_\d(\hy^\beta))$ is said to be outgoing if the vector $\omega_c(\,\cdot\,,\beta)\in V^*$ is not in $\d$.
Thus in most cases, a wall containing the opposite vector $-\omega_c(\,\cdot\,,\beta)$ is particularly outgoing.
A gregarious shard $\Sigma$ can fail to be outgoing, but this only happens when there are no other shards in $\beta^\perp$, or in other words when $\Sigma$ is all of $\beta^\perp$.
To avoid taking the analogy too far, we will not call the unique gregarious shard in $\beta^\perp$ the ``lonely'' shard.
\end{remark}

\begin{prop}\label{unique greg}
If $B$ is acyclic of finite type, then for each positive root $\beta$, the shard $\Sigma_c(\beta)$ is the unique gregarious shard in~$\beta^\perp$.
\end{prop}

The proof of Proposition~\ref{unique greg} requires additional background on shards that we will not give here.
The structure of the proof is an induction on length and rank similar to the proof of Proposition~\ref{camb out}.
However, since we are dealing with shards instead of walls of the $c$-Cambrian fan, we need \cite[Observation~4.7]{sort_camb}, which describes how the decomposition of hyperplanes into shards is affected by the action of $s$.

We rephrase Proposition~\ref{unique greg} as a statement about scattering diagrams.

\begin{cor}
If $B$ is acyclic of finite type, then $\Scat^T(B)$ can be constructed entirely of gregarious walls, with exactly one wall in each hyperplane $\beta^\perp$, where $\beta$ runs over all positive roots.
\end{cor}

For trivial reasons, the property that every wall is gregarious also holds in rank~$2$.

%
%

\section*{Acknowledgments}
Thanks to Gregg Musiker for helpful suggestions and to Tom Sutherland for helpful answers to questions.
Thanks to Dylan Rupel and Salvatore Stella for correcting errors in an earlier version.
Thanks to the anonymous referees for helpful comments and suggestions.

\end{document}